\documentclass[12pt]{amsart}
\usepackage{amsmath,amssymb,amsthm,amscd,esint}

\newtheorem{theorem}{Theorem}[section]
\newtheorem{lemma}[theorem]{Lemma}

\newtheorem{corollary}[theorem]{Corollary}
\newtheorem{proposition}[theorem]{Proposition}

\theoremstyle{definition}
\newtheorem{definition}[theorem]{Definition}

\theoremstyle{remark}
\newtheorem{remark}[theorem]{Remark}

\numberwithin{equation}{section}
\numberwithin{theorem}{section}

\DeclareMathOperator{\vol}{vol} \DeclareMathOperator{\dist}{dist}
 
\DeclareMathOperator{\Hess}{Hess} 
 
\DeclareMathOperator{\diam}{diam}

\newcommand{\Ric}{\mathrm{Ric}}

\newcommand{\be}{\begin{equation}}
\newcommand{\ee}{\end{equation}}
\newcommand{\ba}{\begin{eqnarray}}
\newcommand{\ea}{\end{eqnarray}}
\newcommand{\ban}{\begin{eqnarray*}}
\newcommand{\ean}{\end{eqnarray*}}
\newcommand{\rp}{\rangle}
\newcommand{\lp}{\langle}
\newcommand{\ra}{\rightarrow}
\newcommand{\ID}{\mathrm{ID}}
\newcommand{\IN}{\mathrm{IN}}
\newcommand{\SD}{\mathrm{SD}}
\newcommand{\SN}{\mathrm{SN}}

\begin{document}
\title[Integral Ricci Curvature Bounds]{ Local Sobolev Constant Estimate for  \\
Integral Ricci Curvature Bounds }
\author{Xianzhe Dai}
\address[Xianzhe Dai]{Department of Mathematics, ECNU, Shanghai and UCSB, Santa Barbara CA 93106 \\ email:dai@math.ucsb.edu}
\author{Guofang Wei}
\address[Guofang Wei]{Department of Mathematics, University of California, Santa Barbara CA 93106 \\ email: wei@math.ucsb.edu}
\author{Zhenlei Zhang}
\address[Zhenlei Zhang]{Department of Mathematics, Capital Normal University, China \\ email: zhleigo@aliyun.com}
\date{}
\subjclass{53C20}
\keywords{Integral curvature bounds, Isoperimetric Constant, $L^2$ Hessian estimates.}
\thanks{XD partially supported by the Simons Foundation, NSF and CNSF, GW partially supported by the Simons Foundation and NSF DMS 1506393, ZZ partially supported by NSFC11371256 and NSFC11431009 }
\begin{abstract}
We obtain a local Sobolev constant estimate for integral Ricci curvature, which enables us to extend several important tools such as the maximal principle, the gradient estimate, the heat kernel estimate and the $L^2$ Hessian estimate to manifolds with integral Ricci lower bounds, without the non-collapsing conditions.
\end{abstract}
\maketitle

\section{Introduction}

Integral curvature is a very natural notion as it occurs in diverse situations, for example, the Chern-Gauss-Bonnet formula, the isospectral problem, and numerous variational problems. Moreover, integral curvature bounds have recently been discovered in various geometric situations, such as the $L^2$ bound of the curvature tensor for noncollapsed manifolds with bounded Ricci curvature, and the (almost) $L^4$ bound of the Ricci curvature for the K\"ahler-Ricci flow as well as the (real) Ricci flow (under certain conditions) \cite{Cheeger-Naber, Jiang-Naber, TiZh13, Bamler-Zhang, Simon, TiZh15, Bamler}. In \cite{PeWe97}, the important Laplacian comparison and volume comparison are generalized to integral Ricci lower bound. Combining this with D. Yang's estimate \cite{Ya92} on the local Sobolev constant, the Cheeger-Colding-Naber theory has now been successfully extended to integral Ricci curvature bound in the noncollapsed case, with important consequences \cite{PeWe00, TiZh13}. In the collapsed case a local Sobolev constant estimate was missing. Here we provide the missing piece and extend many of the basic estimates for integral curvature in \cite{PeWe00, TiZh13} to the collapsed case.

For each $x\in M^n$ let $\rho\left( x\right) $ denote the smallest
eigenvalue for the Ricci tensor $\mathrm{Ric}:T_{x}M\rightarrow
T_{x}M,$ and $\Ric_-^H(x) = \left( (n-1)H - \rho
(x)\right)_+ = \max \left\{ 0, (n-1)H - \rho (x) \right\}$, the amount of Ricci curvature lying below $(n-1)H$. Let
\be \| \Ric_-^H \|_{p, R} = \sup_{x\in M} \left( \int_{B\left( x,R\right) } (\mathrm{Ric}_-^H)^{p}\,
  d vol\right)^{\frac 1p}. \ee
Then  $\| \mathrm{Ric}_-^H \|_p $ measures the amount of Ricci
curvature lying below a given bound, in this case,  $(n-1)H$, in the $L^p$ sense. Clearly $\|
\mathrm{Ric}_-^H \|_{p, R} = 0$ iff $\Ric_M \ge (n-1)H$.
It is often convenient to work with the following {\bf scale invariant} curvature quantity (with $H=0$):
\begin{equation} \label{Lp Ricci: average}
k\left(x, p,R\right) = R^2\left( \fint_{B_R\left( x\right) } \rho_-
^{p}\right) ^{\frac{1}{p}}, \  \ \ \ \ \ \ \ \  k\left( p,R\right) = \sup_{x\in M} k\left(x, p,R\right).
\end{equation}

The main result of the paper is
\begin{theorem} \label{local-iso} For $p> n/2$, there exists $\varepsilon = \varepsilon(p,n)>0$ such that if $M^n$ has $k(p,1) \le \varepsilon,$ then for any $x \in M, \  r \le 1$ with $\partial B_1(x) \not= \emptyset$, the normalized Dirichlet isoperimetric constant has the estimate
\begin{equation}\label{isoperimetric: 1}
\ID^*_n (B_{r}(x)) \le  10^{2n+4} r,
\end{equation}
where
$$ \ID^*_n (B_{r}(x)) = \vol(B_{r}(x))^{\frac{1}{n}} \cdot  \sup_{\Omega} \bigg\{\frac{\vol(\Omega)^{1-\frac{1}{n}}}{\vol(\partial\Omega)} \bigg\}.$$
Here the supremum runs over all subdomains  $\Omega\subset B_{r}(x)$ with smooth boundary  and $\partial \Omega \cap \partial B_{r}(x) = \emptyset$.
\end{theorem}

See Section~\ref{D-N-S} for a discussion of isoperimetric constants.

\begin{remark}
The smallness of $k(p,1)$ is necessary. Namely the result is not true if we only assume that $k(p,1)$ is bounded; see Section~\ref{smallness} for detail. Also the result is not true when $p \le \frac n2$  \cite{Aubry}.
\end{remark}

\begin{remark}
 In the presence of the non-collapsing condition $\vol B_r(x) \ge c r^n$, our scale invariant curvature quantity
 $k(x, p, r) \le c^{-1/p} r^{2-\frac np} \|\Ric_-\|_{p, B_r(x)}$, which is always small when $\|\Ric_-\|_{p, B_1(x)}$ is bounded and $r$ is small. This has been very nicely applied in \cite{TiZh13, TiZh15}. Namely when applying  to the study of tangent cones,  with local volume growth, one only needs to assume that $\|\Ric_-\|_{p, B_1(x)}$ is bounded.  Note also that when $k(p,r)$ is small for some $r$, it gives control on $k(p, r)$ for all $r$,  see Remark~\ref{all-r-small} for detail.
\end{remark}

\begin{remark}
In terms of the usual isoperimetric constant, our estimate reads
$$   \sup_{\Omega} \bigg\{\frac{\vol(\Omega)^{1-\frac{1}{n}}}{\vol(\partial\Omega)} \bigg\} \leq \frac{10^{2n+4} r }{\vol(B_{r}(x))^{\frac{1}{n}} }.
$$
We emphasize that it is very important that the volume dependence here is $\vol(B_{r}(x))^{\frac{1}{n}}$.
It is of the right scale invariance, and  corresponds to the optimal Sobolev constant. We note that a local isoperimetric constant estimate is given indenpendently in a recent paper \cite{Rose2016} but with weaker result and under much stronger assumptions.
\end{remark}

From (\ref{normal-Ch}) and (\ref{ID*-SD*}), the theorem above immediately gives
\begin{corollary}
Under the same assumption as in Theorem~\ref{local-iso}, we have the Cheeger's constant $$\ID_\infty (B_{r}(x)) \le  10^{2n+4} r$$ and the Sobolev inequality
\be
\left( \fint_{B_r(x)} f^{\frac{n}{n-1}} \right)^{\frac{n-1}{n}}  \le 10^{2n+4} r  \fint_{B_r(x)} |\nabla f |,   \label{L1-local-sobolev} \ee
for all $f \in C_0^\infty (B_r(x))$ where $r\le 1$.
\end{corollary}

See Definition~\ref{ID-alpha} for the definition of Cheeger's constant. By Cheeger's inequality \cite{Cheeger1970}, the first eigenvalue $\lambda_1 \ge \frac{1}{4\, \ID^2_\infty}$. Thus we also obtain an eigenvalue lower bound. We emphasize  that in the above Sobolev inequality we use the averaged integral (volume normalized).

\begin{remark}
Under the pointwise Ricci curvature lower bound, estimates of the type above (namely local or Dirichlet) for Cheeger's constant and isoperimetric constant  are proved in \cite{Buser1982, An92}. For integral Ricci curvature lower bound, D. Yang \cite{Ya92} obtained a local Sobolev constant estimate under the additional assumption that the manifold is noncollaped, see Theorem~\ref{D-Yang}. Paeng \cite{Paeng2011} proved a local Cheeger's constant estimate for integral Ricci curvature under some strong assumption.
\end{remark}

\begin{remark}
When $M$ is closed, the global (Neumann) normalized isoperimetric constant (see Section~\ref{D-N-S} for definition) for integral Ricci curvature was already obtained in \cite{Ga88}, see Theorem~\ref{Gallot-iso}. The proof for global one does not apply to local one here since it uses a result from geometric measure theory which only works for closed manifolds or domains with convex boundary.
\end{remark}

The local Sobolev inequality enables us to obtain many applications. First we can extend the maximal principle and gradient estimate in \cite{PeWe00} to the collapsed case. Namely we have the following maximal principle.
\begin{theorem} \label{Max-prin}
	Let $M$ be an $n$-dimensional Riemannian manifold, and $p>n/2$. There is an $\varepsilon=\varepsilon \left(n,p\right) >0$ and $C=C\left( n,p,q\right) >1$ such that if $k(p,1) \leq \varepsilon$ and $R\leq 1$ then any function $%
	u:\Omega \subset B\left( x,R\right) $ $\rightarrow \mathbb{R}$ with $\Delta
	u\geq f$ satisfies
	\begin{equation*}
	\sup_{\Omega }u\leq \sup_{\partial \Omega }u+R^{2}\cdot C\cdot \left\|
	f_-\right\|^*_{q,\Omega},
	\end{equation*}
	for any $q>\frac n2.$ Here the normalized $L^q$ norm $\left\|
	f_-\right\|^*_{q,\Omega}$ is introduced at the beginning of the next section.
\end{theorem}

Also we have the gradient estimate.

\begin{theorem} \label{gradient}
	Let $M$ be an $n$-dimensional Riemannian manifold, and $p>n/2$. There is an $\varepsilon \left(
	n,p\right) >0$ and $C\left( n,p\right) >1$ such that if $k(p,1) \leq \varepsilon$ and $R\leq 1$ and $u$ is a function on  $B_1(x)$ satisfying $$\Delta u = f, $$  then
	\be
	\sup_{B_{\frac{R}{2}}(x)}|\nabla u|^2 \leq C(n,p) R^{-2}  \left[ (\|u\|_{2, B_R(x)}^* )^2 + (\|f\|_{2p, B_R(x)}^*)^2\right].
	\ee
\end{theorem}

With the (relative) local Sobolev constant estimate (\ref{L1-local-sobolev}), one gets heat kernel upper bound, see e.g. \cite[(2.17)]{Hesisch-Saloff-Coste2001}. With this and the volume doubling (\ref{volume doubling}), Zhang-Zhu obtained Li-Yau's gradient estimate
\cite{Zhang-Zhu2016}. Hence one has
parabolic Harnack inequality and local Li-Yau heat kernel lower bound, see Theorem~\ref{heat-bounds}. Consequently we derive the following mean value inequality, extending the one in \cite{TiZh13} to the collapsed case.

\begin{theorem} For any integer $n$ and $p>\frac{n}{2}$ there exist $\varepsilon=\varepsilon \left(
	n,p\right) >0$ and $C=C\left( n,p\right) >1$ such that the following holds. Given $M$ a complete $n$-dimensional Riemannian manifold satisfying $k(p,1)\le\varepsilon$, let $u$ be a nonnegative function satisfying
	\[
	\frac{\partial}{\partial t} u \ge \Delta u - f,\]
	where $f$ is a nonnegative space-time function,
	then,  for $q > \frac n2$,
	\be
	\fint_{B_{\frac 12 r} (x)} u(\cdot, 0) d \vol  \le C u(x, r^2) + C(n,p,q)\,  r^2 \sup_{t \in [0, r^2]}  \|f(t)\|^*_{q, B_r(x)}
	\ee
	for all $x \in M$, $r \le 1 $.
\end{theorem}

With the above tools at our disposal, we can then extend the
$L^2$ Hessian estimate for parabolic approximation of Colding-Naber to integral curvature, see Section~\ref{application} for detail. In the noncollapsed case it is established in \cite{TiZh13}, see also \cite{Zhang-Zhu15}.

We expect further applications of our results e.g. to the Cheeger-Colding-Naber theory, which will be discussed in a future paper.

{\em Acknowledgment} The second author would like to thank Qi Zhang for his interest and helpful conversations.  We also would like to thank 
Christian Rose for pointing out lapses in the argument of heat kernel lower bound in the earlier version.  
\section{Preliminary}

In this section we fix notations and recall the previous work \cite{PeWe97}, \cite{PeWe00} that will play a fundamental role here. We also give a review on the isoperimetric and Sobolev constants and their relations,  and introduce the normalized version.

For functions $f$ on M, the $L^p$ norm and normalized $L^p$ norm on a ball $B(x, r) \subset M$ is denoted
\[
\|f\|_{p, B(x,r)} =\left(\int_{B(x, r)}  |f|^p\right)^{\frac{1}{p}}, \ \ \ \ \|f\|_{p, B(x, r)}^*= \left(\fint_{B(x,r)}  |f|^p\right)^{\frac{1}{p}}.
\]
(The notation of the volume form of $g$ is  often omitted in this paper.) $\|f\|_p,  \|f\|_p^*$ denote the norm, normalized norm of $f$ on $M$.

\subsection{Volume Comparison for Integral Curvature}

For simplicity, we state the case when $H= 0$.
Let $M^n$ be a complete Riemannian manifold of dimension $n$. Given $x \in M$, let $r(y) = d(y, x)$ be the distance function and $\psi (y) = \left( \Delta r - \frac{n-1}{r} \right) _{+}$. The classical Laplacian comparison states that if $\Ric_M \ge 0$, then $\Delta r \le \frac{n-1}{r}$, i.e., if $\Ric_- \equiv 0$, then $\psi  \equiv 0$. In \cite{PeWe97} this is generalized to integral Ricci lower bound.

\begin{theorem}[Laplacian and Volume Comparison \cite{PeWe97, PeWe00}]
Let $M^n$ be a complete Riemannian manifold of dimension $n$. If $p>\frac{n}{2}$, then
\be \label{Laplacian-com}
\| \psi \|_{2p. B(x,r)}  \leq \left(\frac{(n-1)(2p-1)}{2p-n} \|
\mathrm{Ric}_- \|_{p, B(x, r)} \right)^{\frac 12}. \ee
Equivalently
\be
\label{Laplacian-com-average-norm}
\| \psi \|^*_{2p. B(x,r)}  \leq \left(\frac{(n-1)(2p-1)}{2p-n} \|
\mathrm{Ric}_- \|^*_{p, B(x, r)} \right)^{\frac 12}. \ee

Consequently we have the following volume comparison estimate: for any $r_2\geq r_1>0$,
\begin{equation}\label{relativevolume: 1}
\bigg(\frac{\vol(B_{r_2}(x))}{r_2^n}\bigg)^{\frac{1}{2p}}-\bigg(\frac{\vol(B_{r_1}(x))}{r_1^n}\bigg)^{\frac{1}{2p}}\leq C(n,p)r_2^{1-\frac{n}{2p}}\left( \|Ric_-\|_{p, B(x, r_2)}\right)^{\frac{1}{2}}.
\end{equation}
In other words,
\begin{equation}\label{relativevolume: 3}
\bigg(\frac{\vol(B_{r_1}(x))}{\vol(B_{r_2}(x))}\bigg)^{\frac{1}{2p}}\geq \bigg(\frac{r_1}{r_2}\bigg)^{\frac{n}{2p}}
\bigg[1-C(n,p)\left( k(x,p,r_2) \right)^{\frac{1}{2}}\bigg],
\end{equation}
where $C(n,p)$ is a  constant depending on $n,p$.
Hence there exists $\varepsilon_0=\varepsilon_0(p,n)>0$ such that, if $k(x, p, r_0)  \le\varepsilon_0,$
 then
\begin{equation}\label{volume doubling}
\frac{\vol(B_{r}(x))}{\vol(B_{r_0}(x))}\ge\frac{1}{2}\bigg(\frac{r}{r_0}\bigg)^{n},\,\forall r<r_0.
\end{equation}
\end{theorem}

\begin{remark}  \label{all-r-small}
As pointed out in \cite[Section 2.3]{PeWe00},  if $k(x,p, r_2) \le \varepsilon_0$ for the $\varepsilon_0$ above, then (\ref{volume doubling}) implies,
\begin{equation}  \label{k-2-1}
k(x,p, r_1)  \leq 2^{1/p} \left(
\frac{r_{1}}{r_{2}}\right) ^{2-\frac{n}{p}}\cdot k(x,p, r_2) \le 2^{1/p} k(x,p, r_2),      \ \ \  \forall  r_1 \le r_2.
\end{equation}
Hence $k(x,p,r_1) \ra 0$ as $r_1 \ra 0$ and $k(x,p,r_1) \le \varepsilon_0(p,n)$ when $r_1 \le 2^{\frac{1}{n-2p}} r_2$. On the other hand, when $k(p,r_1) \le \varepsilon_0(p,n)$, then \be
k(p, r_2) \le 2^{\frac{n+1}{p}} \left(
\frac{r_{2}}{r_{1}}\right)^2 k(p, r_1) \ \ \mbox{for all} \   r_2 \ge r_1.   \label{k-1-2}  \ee
 Hence when $k(p,r)$ is small for some $r$, it gives control on $k(p, r)$ for all $r$.
\end{remark}

 Note also that if one
has a lower bound for the Ricci curvature $\mathrm{Ric}\geq \left(
n-1\right) H$ then the quantity $k\left( p,R\right) $
will be small for sufficiently small $R.$ Namely general lower bound can be reduced to zero lower bound in the local analysis.

\subsection{Dirichlet and Neumann Isoperimetric and Sobolev Constants}
\label{D-N-S}
In this subsection we review the definitions of the isoperimetric and Sobolev constants and their relations, and introduce the normalized form. For details, see \cite{Li, Chavel}, though we use a different convention here.

\begin{definition}  \label{ID-alpha}
For a complete noncompact Riemannian manifold $M^n$ or a compact Riemannian manifold $M^n$ with $\partial M \not= \emptyset$,  for $n \le \alpha \le \infty$,  the Dirichlet (also referred as local) $\alpha$-isoperimetric constant of $M$ is defined by
\[
\ID_\alpha (M) = \sup_{\Omega}  \frac{ \vol (\Omega)^{1-\frac{1}{\alpha}} }{\vol (\partial \Omega )}, \]
where $\Omega$ is an open submanifold of $M$ with compact closure and smooth boundary such that $\partial \Omega \cap \partial M = \emptyset$.
\end{definition}

When $\alpha = n$, $\ID_n (M)$ is scale invariant. When $\alpha = \infty$, this is Cheeger's constant \cite{Cheeger1970}, which scales like $\vol ^{1/n}$. The Dirichlet $\alpha$-isoperimetric constant controls the local volume growth: for given a geodesic ball $B(x, r) \subset M$,  $\vol B(x, r) \ge \left( \frac{r}{\alpha \, \ID_\alpha (M)} \right)^\alpha$ for $n \le \alpha < \infty$.

\begin{definition}
The Dirichlet $\alpha$-Sobolev constant of $M$ is defined by
\[
\SD_\alpha (M) = \sup_{f} \frac{\|f\|_{\frac{\alpha}{\alpha-1}} }{ \|\nabla f\|_1}, \]
where $f$ ranges over $C_c^\infty (M)$.
\end{definition}

\begin{definition}
When $M$ is compact with or without boundary, the Neumann $\alpha$-isoperimetric constant of $M$ is defined by
\[
\IN_\alpha (M) = \sup_{\Gamma}  \frac{ \min \{ \vol (D_1), \vol (D_2) \}^{1-\frac{1}{\alpha}} }{\vol (\Gamma )}, \]
where $\Gamma$ varies over compact $(n-1)$-dim submanifold of $M$ which divide $M$ into two disjoint open submanifolds $D_1, D_2$ of $M$.
\end{definition}

From the definition, if $\Omega \subset M$,  $\partial \Omega \cap \partial M = \emptyset$, and $\vol (\Omega) \le \frac 12 \vol (M)$, then \be
\ID_\alpha (\Omega) \le \IN_\alpha (M).   \label{ID-IN} \ee

\begin{definition}
The Neumann $\alpha$-Sobolev constant of $M$ is defined by
\[
\SN_\alpha (M) = \sup_f \frac{\inf_{a \in \mathbb R}\|f-a\|_{\frac{\alpha}{\alpha-1}} }{ \|\nabla f\|_1}, \]
where $f$ ranges over $C^\infty (M)$.
\end{definition}

\begin{theorem}[\cite{Federer-Fleming1960, Cheeger1970}, see also \cite{Li, Chavel}]  \label{ID-SD}
For all $n \le \alpha \le \infty$,
\[
\ID_\alpha (M) = \SD_\alpha (M), \ \  \IN_\alpha (M) \ge \SN_\alpha (M) \ge  \frac 12 \, \IN_\alpha (M). \]
\end{theorem}

For convenience we consider the normalized  Dirichelet and Neumann $\alpha$-isoperimetric and $\alpha$-Sobolev constants:
$$ \ID^*_\alpha (M) = \ID_\alpha (M) \, \vol (M)^{1/\alpha}, \ \ \  \ \SD^*_\alpha (M) = \SD_\alpha (M) \,\vol (M)^{1/\alpha}, $$
$$\IN^*_\alpha (M) = \IN_\alpha (M) \, \vol (M)^{1/\alpha}, \ \ \  \ \SN^*_\alpha (M) = \SN_\alpha (M) \, \vol (M)^{1/\alpha}. $$
Observe that \be  \label{normal-Ch}
\ID^*_\alpha (M) \ge \ID_\infty (M), \ \ \  \IN^*_\alpha (M) \ge  \IN_\infty (M), \ee
 and
\[
\SD^*_\alpha (M) = \sup_{f} \frac{\|f\|^*_{\frac{\alpha}{\alpha-1}} }{ \|\nabla f\|^*_1},  \ \mbox{where}\  f \ \mbox{ranges over }  C_c^\infty (M),  \]
\[
\SN^*_\alpha (M) = \sup_f \frac{\inf_{a \in \mathbb R}\|f-a\|^*_{\frac{\alpha}{\alpha-1}} }{ \|\nabla f\|^*_1},  \  \mbox{where} \ f \ \mbox{ranges over}\  C^\infty (M). \]

By Theorem~\ref{ID-SD}, we have
\be  \label{ID*-SD*}
\ID^*_\alpha (M) = \SD^*_\alpha (M).
\ee

These normalized quantities are very useful in studying the collapsed case, see below. They are used  in \cite{Wei-Ye} in proving a Neumann type maximal principle without volume lower bound.

In \cite[Theorem 7.4]{Ya92} D. Yang obtained a Dirichlet isoperimetric constant estimate in the non-collapsing case when $\|Ric_-\|^*_p$ is small. Namely
\begin{theorem} \label{D-Yang}
Given $p > n/2$ and $v>0$, there is an $\varepsilon \left(
n,p,v\right) >0$ such that if $B_1\left(x\right) \subset M^{n}$ has $\vol B_1\left( x\right) \geq v$ and $k\left( p,1\right) \leq \varepsilon,$  then   $\ID_n (B_{\frac 12}(x)) \le C\left( n,p, v\right)$.
\end{theorem}

For closed Riemannian manifold $M^n$, Gallot \cite[Theorem 3]{Ga88} showed that the normalized Neumann $\alpha$-isoperimetric constant is bounded from above when $\diam(M)^2\|Ric_-\|^*_p$ is small ($\le \epsilon(n,p)$) for $p > n/2$, and $\alpha >n$.  Petersen-Sprouse \cite{Petersen&Sprouse}  obtained the bound for $\alpha =n$. Namely
\begin{theorem}  \label{Gallot-iso}
Given $p> n/2$ and $D>0$,  there is an $\varepsilon \left(
n,p,D\right) >0$ such that if $\diam M^{n} \le D$ and $ \|Ric_-\|^*_p \leq
\varepsilon,$  then   $\IN^*_{n} (M) \le C\left( n,p, D\right)$.
\end{theorem}

\section{Local isoperimetric constant estimate for closed manifolds}

For the local analysis, we need local (Dirichlet) Sobolev constant bound. From (\ref{ID-IN}), we automatically get a local estimate when the volume of the domain is small relative to the whole manifold. We show that  the measure can only have small concentration whenever $\|\Ric_-\|_p^*$ is small.

\begin{proposition}  \label{global-small}
Suppose $\diam(M)=D$. There exists $\varepsilon=\varepsilon(n,p)>0$ such that if
\begin{equation}
D^2 \|\Ric_-\|^*_p\le  \varepsilon,
\end{equation}
then for any $a\le a_0$ where $a_0=a_0(n)$ solves
\begin{equation}
\frac{\frac{1}{2}-a_0}{\frac{1}{2}+a_0}=\bigg(\frac{3}{4}\bigg)^{\frac{1}{n}}
\end{equation}
we have
\begin{equation}
\vol(B_{aD}(x))\le\frac{1}{2}\vol(M),\,\forall x\in M.
\end{equation}
\end{proposition}
\begin{proof}
For any $x\in M$ we choose a dual point $x'\in M$ with $\dist(x,x')=\frac{D}{2}$. Then, for any radius $r<\frac{D}{2}$,
\begin{equation}\label{e: 1}
\frac{\vol(B_r(x))}{\vol(M)}\le\frac{\vol(B_r(x))}{\vol(B_{\frac{D}{2}+r}(x'))}\le 1-\frac{\vol(B_{\frac{D}{2}-r}(x'))}{\vol(B_{\frac{D}{2}+r}(x'))}.
\end{equation}
Therefore it suffices to show that for $r=aD$ with $a\le a_0$, the last term above is greater than or equal to  $\frac{1}{2}$.
By (\ref{relativevolume: 3}) the last term can be estimated  as follows
\begin{equation}\label{e: 2}
\bigg(\frac{\vol(B_{\frac{D}{2}-r}(x'))}{\vol(B_{\frac{D}{2}+r}(x'))}\bigg)^{\frac{1}{2p}}\geq \bigg(\frac{\frac{D}{2}-r}{\frac{D}{2}+r}\bigg)^{\frac{n}{2p}}
\bigg[1-C(n,p)\left(k(x',p,\tfrac{D}{2}+r)\right)^{\frac 12}\bigg].
\end{equation}
If $k(x',p,D) = D^2 \|\Ric_-\|^*_p \le \varepsilon_0$, by (\ref{k-2-1}),
$$k(x',p,\tfrac{D}{2}+r)   \le 2^{1/p} k(x',p,D) = 2^{1/p} D^2 \|\Ric_-\|^*_p.$$
Hence if we assume that
\begin{equation}\label{e: 2.5}
D^2 \|\Ric_-\|^*_p \le \varepsilon \le \varepsilon_0,
\end{equation}
 then
\begin{equation}\label{e: 3}
\frac{\vol(B_{\frac{D}{2}-r}(x'))}{\vol(B_{\frac{D}{2}+r}(x'))} \geq \bigg(\frac{\frac{D}{2}-r}{\frac{D}{2}+r}\bigg)^n
\bigg[1-C(n,p) 2^{1/2p} \varepsilon^{\frac 12}\bigg]^{2p}.
\end{equation}
Plug in  $r=aD$ with $a\le a_0$, the choice of $a_0$ implies that
$$\bigg(\frac{\frac{D}{2}-aD}{\frac{D}{2}+aD}\bigg)^n\ge\tfrac{3}{4}.$$
Now set \[
\tfrac 34 \bigg[1-C(n,p) 2^{1/2p} \varepsilon^{\frac 12}\bigg]^{2p} \ge \frac 12.\]
Clearly there exists $\varepsilon(n,p)$ such that  this holds for all $\varepsilon \le \varepsilon(n,p)$.
\end{proof}

Combining this with Theorem~\ref{Gallot-iso} and (\ref{ID-IN}), we have
\begin{theorem}
Given $p> n/2$ and $D>0$,  there is an $\varepsilon \left(
n,p,D\right) >0, r_0 = r_0(n)$ such that if $\diam M^{n} \le D$ and $ \|Ric_-\|^*_p \leq
\varepsilon,$  then   $\ID^*_{n} (B_r(x)) \le C\left( n,p, D\right)$ for all $x\in M$ and $r\le r_0$.
\end{theorem}

Similarly we have a  local version of Proposition~\ref{global-small} which will be needed in the next section.

\begin{theorem}  \label{loc-vol-small}
There exists $\varepsilon=\varepsilon(n,p)>0$ and $r_0=r_0(n)>0$ such that the following holds. Let $(M,g)$ be a complete noncompact Riemannian manifold satisfying $k(p,1) \le\varepsilon,$
then we have
\begin{equation}
\frac{\vol(B_{r_0}(x))}{\vol(B_1(x))}\le\frac{1}{2},\,\forall x\in M.
\end{equation}
\end{theorem}
\begin{proof}
For any $x \in M, r <\frac 13$,  choose a point $x'$ with $d=\dist(x,x')=\frac{1-r}{2}\ge\frac{1}{3}$. Then we have
$$B_r(x)\subset B_1(x)\backslash B_{d-r}(x')\subset B_{d+r}(x')\subset B_1(x).$$
As above we calculate
$$\frac{\vol(B_r(x))}{\vol(B_1(x))}\le\frac{\vol(B_r(x))}{\vol(B_{d+r}(x'))}\le 1-\frac{\vol(B_{d-r}(x'))}{\vol(B_{d+r}(x'))}.$$
To estimate the last term recall that
\begin{equation}\label{e: 21}
\bigg(\frac{\vol(B_{d-r}(x'))}{\vol(B_{d+r}(x'))}\bigg)^{\frac{1}{2p}}\geq \bigg(\frac{d-r}{d+r}\bigg)^{\frac{n}{2p}}
\bigg[1-C(n,p) \big( k(x',p, d+r) \big)^{\frac 12} \bigg].
\end{equation}
Since $d+r \le 1$, by (\ref{k-2-1}), if $k(x',p,1) \le \varepsilon_0$,  we have $k(x',p,d+r)   \le 2^{1/p} k(x',p,1)$. Hence when $\varepsilon \le \varepsilon_0$,  we get
$$\bigg(\frac{\vol(B_{d-r}(x'))}{\vol(B_{d+r}(x'))}\bigg)^{\frac{1}{2p}}\geq\bigg(\frac{d-r}{d+r}\bigg)^{\frac{n}{2p}}
\left(1-C(n,p) 2^{1/2p}\varepsilon^{\frac{1}{2}}\right). $$
Now we choose $a_0$ such that
$$\frac{1-a_0}{1+a_0}=\left(\frac{3}{4}\right)^{\frac{1}{n}},$$ then for any $r\le\frac{1}{3}a_0$,  we have, since $d \ge\frac{1}{3}$,
$$ \bigg(\frac{d-r}{d+r}\bigg)^n \ge \frac 34. $$
Choose $\varepsilon \le \varepsilon_0$ such that
\begin{equation}\label{epsilon: 2}
\left(1-C(n,p) 2^{1/2p} \varepsilon^{\frac{1}{2}} \right)^{2p}\ge\frac{2}{3}.
\end{equation}
Then
$$\frac{\vol(B_r(x))}{\vol(B_1(x))}\le 1-\bigg(\frac{d-r}{d+r}\bigg)^{n}
\left(1-C(n,p)2^{1/2p}\varepsilon^{\frac{1}{2}}\right)^{2p}\le1-\frac{3}{4}\cdot\frac{2}{3}=\frac{1}{2}.$$
The proof is complete by choosing $r_0=\frac{1}{3}a_0$ and any $0<\varepsilon \le \varepsilon_0$ satisfying (\ref{epsilon: 2}).
\end{proof}

\section{Local isoperimetric constant estimate for complete manifolds}

In this section we first obtain an estimate on the weak Cheeger's constant with an error using Laplacian comparison for integral curvature and an idea of Gromov \cite[Page 9-10]{Gr}.  This will then enable us to prove Theorem~\ref{local-iso} by using a covering argument of Anderson \cite{An92}.

Recall the following lemma of Gromov which is stated for closed manifold in \cite{Gr}, but also works for complete manifolds.
\begin{lemma}[\cite{Gr}]
Let $M^n$ be a complete Riemannian manifold and $H$ be any hypersurface dividing $M$ into two parts $M_1,M_2$. For any Borel subsets $W_i\subset M_i$, there exists $x_1$ in one of $W_i$, say $W_1$, and a subset $W$ in the  other one, $W_2$, such that
\begin{equation}
\vol(W)\ge\frac{1}{2}\vol(W_2)
\end{equation}
and any $x_2\in W$ has a unique minimal geodesic connecting to $x_1$ which intersects $H$ at some $z$ such that
\begin{equation}
\dist(x_1,z)\ge\dist(x_2,z).
\end{equation}
\end{lemma}

Using Laplacian comparison estimate we have
\begin{lemma}
Let $H$, $W$ and $x_1$ be as in above lemma. Then
\begin{equation}
\vol(W)\le2^{n-1}D\bigg[\vol(H')+\vol(B_D(x_1))\|Ric_-\|_{p,B_D(x_1)}^{*\frac{1}{2}}\bigg]
\end{equation}
where $D=\sup_{x\in W}\dist(x_1,x)$ and $H'$ is the set of intersection points with $H$ of geodesics $\gamma_{x_1,x}$ for all $x\in W$.
\end{lemma}
\begin{proof}
Let $\Gamma\subset S_{x_1}$ be the set of unit vectors such that $\gamma_v=\gamma_{x_1,x_2}$ for some $x_2\in W$. We compute the volume in the polar coordinate at $x_1$. Write $dv=\mathcal{A}(\theta,t)d\theta\wedge dt$ in the polar coordinate $(\theta,t)\in S_{x_1}\times\mathbb{R}^+$. Recall that \cite{PeWe97}
$$\frac{\partial}{\partial t}\frac{\mathcal{A}}{t^{n-1}}\le\psi\frac{\mathcal{A}}{t^{n-1}}$$
where $\psi=\max(0,\triangle r(\theta,t)-\frac{n-1}{t})$ denotes the error term of Laplacian comparison. We thus have
\begin{equation}
\mathcal{A}(\theta,r)\le 2^{n-1}\mathcal{A}(\theta,t)+2^{n-1}\int_t^r\psi(\theta,s)\mathcal{A}(\theta,s)ds,\  \forall \tfrac{r}{2}\le t\le r.
\end{equation}
For any $\theta\in\Gamma$, let $r(\theta)$ be the radius such that $\exp_{x_1}(r\theta)\in H$. Then, by assumption, $$W\subset\{\exp_{x_1}(r\theta)|\theta\in\Gamma,\,r(\theta)\le r\le 2r(\theta)\}.$$
Thus,
\begin{eqnarray}
\vol(W)&\le&\int_\Gamma\int_{r(\theta)}^{2r(\theta)}\mathcal{A}(\theta,t)dtd\theta\nonumber\\
&\le&2^{n-1}\int_{\Gamma}\int_{r(\theta)}^{2r(\theta)}\bigg(\mathcal{A}(\theta,r(\theta))
+\int_{r(\theta)}^t\psi(\theta,s)\mathcal{A}(\theta,s)ds\bigg)d\theta dt\nonumber\\
&\le&2^{n-1}D\int_{\Gamma}\mathcal{A}(\theta,r(\theta))d\theta
+2^{n-1}D\int_{\Gamma}\int_0^D\psi(\theta,t)\mathcal{A}(\theta,t)d\theta dt\nonumber
\end{eqnarray}
On the other hand,
$$\vol(H')=\int_{\Gamma}\frac{\mathcal{A}(\theta,r(\theta))}{\cos\alpha(\theta)}d\theta
\ge\int_{\Gamma}\mathcal{A}(\theta,r(\theta))d\theta$$
where $\alpha(\theta)$ is the angle between $H$ and radial geodesic $\exp_{x_1}(t\theta)$. Thus,
$$\vol(W)\le 2^{n-1}D\vol(H')+2^{n-1}D\bigg(\int_\Gamma\int_0^D\psi^{2p}\mathcal{A}d\theta dt\bigg)^{\frac{1}{2p}}
\bigg(\int_\Gamma\int_0^D\mathcal{A}d\theta dt\bigg)^{1-\frac{1}{2p}}.$$
Through the Laplacian estimate (\ref{Laplacian-com}) we get
\begin{equation}
\vol(W)\le 2^{n-1}D\vol(H')+2^{n-1}D\vol(B_D(x_1))^{1-\frac{1}{2p}}\bigg(\int_{B_D(x_1)}|Ric_-|^pdv\bigg)^{\frac{1}{2p}}
\end{equation}
the required estimate.
\end{proof}

Now we can obtain an estimate on the weak Cheeger's constant with an error.
\begin{corollary}
Let $H$ be any hypersurface dividing $M$ into two parts. For any ball $B=B_r(x)$ we have
\begin{eqnarray}
&&\min\big(\vol(B\cap M_1),\vol(B\cap M_2)\big)\nonumber\\
&\le& 2^{n+1}r\bigg[\vol(H\cap B_{2r}(x))+\vol(B_{2r}(x))\|Ric_-\|_{p,B_{2r}(x)}^{*\frac{1}{2}}\bigg].
\end{eqnarray}
\end{corollary}
\begin{proof}
Put $W_i=B\cap M_i$ in above lemma and notice that $D\le2r$ and $H'\subset H\cap B_{2r}(x)$.
\end{proof}

\begin{corollary}  \label{iso-equa-ball}
Given a hypersurface $H$ dividing $M^n$ into two parts, there exists $\varepsilon = \varepsilon(p,n)$ such that if $k(x,p,1) \le \varepsilon$, then for a metric ball $B=B_r(x)$, $r\le\frac{1}{2}$, which is divided equally by $H$, we have
\begin{eqnarray}  \label{iso-equ-ball}
\vol(B_r(x))\le 2^{n+3}r\vol(H\cap B_{2r}(x)).
\end{eqnarray}
\end{corollary}
\begin{proof}
The previous corollary gives
$$\vol(B)\le 2^{n+2}r\bigg[\vol(H\cap B_{2r}(x))+\vol(B_{2r}(x))\bigg(\fint_{B_{2r}(x)}|Ric_-|^pdv\bigg)^{\frac{1}{2p}}\bigg].$$
If $k(x,p,1) \le 2^{-1/p} \varepsilon_0$, by (\ref{k-2-1}), we have $k(x,p,r) \le \varepsilon_0$ for all $r\le 1$. Hence by (\ref{volume doubling}),
$$\vol(B_{2r}(x))\le 2^{n+1}\vol(B).$$
Again if $k(x,p,1) \le 2^{-1/p} 2^{-2(2n+3)}$, then $k(x,p,r) \le 2^{-2(2n+3)}$ for all $r\le 1$.
Hence
$$\vol(B)\le 2^{n+2}r\big(\vol(H\cap B_{2r}(x))+2^n  2^{-(2n+3)} \vol(B)r^{-1}\big), $$
which gives
$$\vol(B)\le2^{n+3}r\vol(H\cap B_{2r}(x)).$$
Therefore choosing $\varepsilon = \min \{ 2^{-1/p} \varepsilon_0, 2^{-1/p} 2^{-2(2n+3)} \}$ suffices.
\end{proof}

This estimate and volume doubling gives an estimate on the local isoperimetric constant via Vitali Covering Lemma.

\begin{proof}[Proof of Theorem \ref{local-iso}]
First of all we show that the isoperimetric constant estimate (\ref{isoperimetric: 1}) holds for some small radius $r_0=r_0(n)$, under the assumption $k(p,1) \le \varepsilon_1$ for some small constant $\varepsilon_1=\varepsilon_1(p,n)$. By Theorem \ref{loc-vol-small}, we may assume that $\varepsilon_1$ is chosen such that there exists $r_0=r_0(n)$ with $\frac{\vol(B_{2r_0}(x))}{\vol(B_{\frac{1}{10}}(x))}\le\frac{1}{2},\,\forall x\in M$. Now given any $y \in M$, let $\Omega$ be a smooth subdomain of $B_{r_0}(y)$. We may assume that $\Omega$ is connected and its boundary $H=\partial\Omega$ divides $M$ into two parts $\Omega$ and $\Omega^c$. For any $x\in\Omega$, let $r_x$ be the smallest radius such that
$$\vol(B_{r_x}(x)\cap\Omega)=\vol(B_{r_x}(x)\cap\Omega^c)=\frac{1}{2}\vol(B_{r_x}(x)).$$ Since $\Omega \subset B_{2r_0}(x)$ and $\vol(B_{2r_0}(x)) \le\frac{1}{2}  \vol(B_{\frac{1}{10}}(x))$, we have  $r_x\le\frac{1}{10}$. Take $\varepsilon_1$ as in Corollary~\ref{iso-equa-ball}, then by (\ref{iso-equ-ball})
\begin{equation}\label{hypersurface: 11}
\vol(B_{r_x}(x))\le 2^{n+3}r_x\vol(H\cap B_{2r_x}(x)).
\end{equation}
The domain $\Omega$ has a covering
$$\Omega\subset\bigcup_{x\in\Omega}B_{2r_x}(x).$$
By Vitali Covering Lemma, cf. \cite[Section 1.3]{LiYa}, we can choose a countable family of disjoint balls $B_i=B_{2r_{x_i}}(x_i)$ such that $\cup_i B_{10r_{x_i}}(x_i) \supset \Omega$.
Moreover, we assume $\varepsilon_1$ is chosen such that $k(p,r) \le \varepsilon_0$ for all $r\le 1$, then by the volume doubling property (\ref{volume doubling}),
$$\frac{\vol(B_{2r_x}(x))}{\vol(B_{10r_x}(x))}\ge\frac{1}{2\cdot 5^n}. $$ Hence
\begin{equation}\nonumber
\sum_i\vol(B_i)\ge\frac{1}{2\cdot 5^n}\sum_i\vol(B_{10r_{x_i}}(x_i))\ge\frac{1}{2\cdot 5^n}\vol(\Omega).
\end{equation}
Applying the volume doubling property (\ref{volume doubling}) again gives
\begin{equation}
\sum_i\vol(B_{r_{x_i}}(x_i))\ge\frac{1}{4\cdot 10^n}\vol(\Omega).
\end{equation}
Moreover, since the balls $B_i$ are disjoint, combining with (\ref{hypersurface: 11}) gives,
\begin{equation}
\vol(\partial\Omega)\ge\sum_i\vol(B_i\cap H)\ge2^{-n-3}\sum_ir_{x_i}^{-1}\vol(B_{r_{x_i}}(x_i)).
\end{equation}
These two estimates lead to
\begin{eqnarray}
\frac{\vol(\Omega)^{\frac{n-1}{n}}}{\vol(\partial\Omega)}
&\le&10^{n-1}2^{n+5}\frac{\big(\sum_i\vol(B_{r_{x_i}}(x_i))\big)^{\frac{n-1}{n}}}{\sum_ir_{x_i}^{-1}\vol(B_{r_{x_i}}(x_i))}
\nonumber\\
&\le& 10^{2n+4}\frac{\sum_i\vol(B_{r_{x_i}}(x_i))^{\frac{n-1}{n}}}{\sum_ir_{x_i}^{-1}\vol(B_{r_{x_i}}(x_i))}\nonumber\\
&\le&10^{2n+4}\sup_i\frac{\vol(B_{r_{x_i}}(x_i))^{\frac{n-1}{n}}}{r_{x_i}^{-1}\vol(B_{r_{x_i}}(x_i))}\nonumber\\
&=&10^{2n+4}\sup_i\bigg(\frac{r_{x_i}^n}{\vol(B_{r_{x_i}}(x_i))}\bigg)^{\frac{1}{n}}.\nonumber
\end{eqnarray}
On the other hand, since  $\dist(y,x_i)\le r_0$, we have
$$B_{r_0}(y)\subset B_{2r_0}(x_i).$$ Now $r_{x_i}\le\frac{1}{10}$,  applying the volume doubling property (\ref{volume doubling}) again,
$$\vol(B_{r_{x_i}}(x_i))\ge\frac{10^n r_{x_i}^n}{2}\vol(B_{\frac{1}{10}}(x_i))\ge\frac{10^nr_{x_i}^n}{2}\vol(B_{r_0}(y)).$$
Substituting into above calculation we get
$$\frac{\vol(\Omega)^{\frac{n-1}{n}}}{\vol(\partial\Omega)}\le
\frac{10^{2n+4}}{\vol(B_{r_0}(y))^{\frac{1}{n}}},$$
the desired estimate.

We next make a scaling argument to show that the estimate (\ref{isoperimetric: 1}) remains hold for any radius $r \le 1$, under the assumption $k(p,1) \le \varepsilon_2$ for a smaller constant $\varepsilon_2=\varepsilon_2(p,n)>0$. Put $r_1=\frac{r}{r_0}\le \frac{1}{r_0}$. After a scaling, it is sufficient to check that
$$k(p,r_1) \le \varepsilon_1.$$  Choose $\varepsilon_2$ such that $\varepsilon_2 \le \varepsilon_0$,  so (\ref{volume doubling}) holds for all $r \le 1$. Now if $r_1\le 1$, by  (\ref{k-2-1})
$$k(p, r_1) \le  2^{1/p} k(p,1) \le 2^{1/p} \varepsilon_2.$$
On the other hand, if $1\le r_1\le \frac{1}{r_0}$, then by (\ref{k-1-2}),
\[
k(p,r_1) \le  2^{\frac{n+1}{p}} r_1^2 \, k(p,1) \le   2^{\frac{n+1}{p}} r_0^{-2} \varepsilon_2.\]
Combining the two cases we can choose $\varepsilon_2=\varepsilon_2(p,n)$ as
$$\varepsilon_2  =  \min\{ 2^{-\frac{1}{p}} \varepsilon_1, 2^{-\frac{n+1}{p}}  r_0^{2}  \varepsilon_1, \varepsilon_0\}.$$
The theorem is now proved by setting $\varepsilon=\varepsilon_2$.
\end{proof}

Combining Theorem~\ref{local-iso} with (\ref{ID*-SD*}), we have
\begin{corollary}
If $k(p, 1) \le \varepsilon$ for the $\varepsilon$ in Theorem~\ref{local-iso},  then,
\begin{equation}
 \|f\|^*_{\frac{n}{n-1}, B_1(x)}  \le 10^{2n+4}  \|\nabla f\|_{1, B_1(x)}^*,\,\forall f\in C_0^\infty(B_1(x)),  \label{L1-Sobolev}
\end{equation}
\end{corollary}

 Applying (\ref{L1-Sobolev}) to $f^{\frac{2(n-1)}{n-2}}$ and using the H\"older inequality gives
\begin{equation}
 \|f\|^*_{\frac{2n}{n-2}, B_1(x)}  \le \frac{2(n-1)}{n-2} 10^{2n+4}  \|\nabla f\|_{2, B_1(x)}^*,\,\forall f\in C_0^\infty(B_1(x)).  \label{L2-Sobolev}
\end{equation}
This is essential in the applications.

By a scaling argument, we have
\begin{corollary}
If $k(p, 1) \le \varepsilon$ for the $\varepsilon$ in Theorem~\ref{local-iso}, then, for any $r\le 1$,
\begin{equation}
\|f\|_{\frac{n}{n-1},B_r(x)}^*\le C(n)r\|\nabla f\|_{1,B_r(x)}^*,\,\forall f\in C_0^\infty(B_r(x)),
\end{equation}
and
\begin{equation}
\|f\|_{\frac{2n}{n-2},B_r(x)}^*\le C(n)r\|\nabla f\|_{2,B_r(x)}^*,\,\forall f\in C_0^\infty(B_r(x)).  \label{L2-Sobolev-r}
\end{equation}
\end{corollary}

\begin{corollary}
If $k(p, 1) \le \varepsilon$ for the $\varepsilon$ in Theorem~\ref{local-iso},  then, for any $r\le 1$, the first eigenvalue of Dirichlet Laplace has lower bound
\begin{equation}
\lambda_1(B_r(x))\ge C(n)^{-1}r^{-2}.
\end{equation}
\end{corollary}
\begin{proof}
Suppose $\triangle f=-\lambda f$ for some $\lambda>0$ and $f$ with $\fint f^2dv=1$ and $f=0$ on $\partial B_r(x)$. Then
$$1=\fint_{B_r(x)}f^2dv\le\bigg(\fint_{B_r(x)}f^{\frac{2n}{n-2}}\bigg)^{\frac{n-2}{n}}\le C(n)r^2\fint_{B_r(x)}|\nabla f|^2=C(n)r^2\lambda.$$
Thus $\lambda\ge C(n)^{-1}r^{-2}$ for any eigenvalue $\lambda>0$.
\end{proof}

\section{Applications}
\label{application}
 With this new local Sobolev constant estimate many of  the results for integral curvature in \cite{PeWe00, TiZh13} can be easily extended to the collapsed case. In particular, we have maximum principle, gradient estimate for harmonic function and heat kernel, excess estimate, $L^2$ estimate for the Hessian of the harmonic and parabolic approximation of the distance function.

Denote $C_s(\Omega)$ the normalized local Soboleve constant of $\Omega \subset M^n$,
\be
\| f\|_{\frac{2n}{n-2}, \Omega}^* \le C_s(\Omega) \|\nabla f\|_{2, \Omega}^*,  \ \ \forall f \in C_0^\infty (\Omega).  \label{Sobolev}
\ee
Note that $C_s(\Omega)$ scales like diameter.

Recall the following maximal principle \cite[Corollary 3.2]{PeWe00}.
\begin{theorem}
Let $M$ be an $n$-dimensional Riemannian manifold, and $p>n/2$. For  any function $u:\Omega \subset M$ $\rightarrow \mathbb{R}$ with $\Delta u\geq -f$, where $f$ is non-negative on $\Omega$, we have
\begin{equation*}
\sup_{\Omega }u\leq \sup_{\partial \Omega }u+C(n,p) \cdot  C^2_s(\Omega) \cdot \left\|
f\right\|^*_{p,\Omega}.
\end{equation*}
\end{theorem}

Combining this with (\ref{L2-Sobolev-r}) gives Theorem~\ref{Max-prin}.

Now we derive the following gradient estimate.
\begin{theorem}  \label{gradient-g}
Let $M$ be an $n$-dimensional Riemannian manifold, and $p>n/2$.   If $u$ is a function on  $B_R(x)$ satisfying $$\Delta u = f, $$  then
\ban
\lefteqn{ \sup_{B_{\frac{R}{2}}(x)}|\nabla u|^2 \leq C(n,p) R^{-2}  \cdot \frac{\vol B_R(x)}{\vol B_{\frac 34 R}(x)}  \cdot \left[ (\|u\|_{2, B_R(x)}^* )^2 + (\|f\|_{2p, B_R(x)}^*)^2\right] } \\
 & & \cdot \left[ \left( R^{-2} C^2_s(B_R(x))\right)^{\frac{2p}{2p-n}} \left(1+k(p,R)^{\frac{2p}{2p-n}} \right)  +  R^{-2} C^2_s \left(1+ R^{-2} C_s^2  \, k(p,R) \right) \right]^{n/2}
\ean
\end{theorem}
The estimate follows from the standard Nash-Moser iteration, by using the $L^p$ integrability of $\Ric$ and $f^2$ for $p>\frac{n}{2}$. On the other hand, as we do not assume the harmonicity of $u$ (i.e. $f=0$), and Ricci curvature pointwise lower bound, the proof requires several extra estimates and the Laplacian comparison estimate (\ref{Laplacian-com-average-norm}). This full general version is often needed in applications.  Since a proof is not in the literature, we give a detailed proof here.
\begin{proof}
By scaling we may assume $R=1$. Recall the Bochner formula,
\begin{equation}\label{Bochnerformula}
\frac{1}{2}\triangle|\nabla u|^2=|\Hess u|^2+ \lp \nabla u, \nabla f \rp + \Ric(\nabla u,\nabla u)\ge \lp \nabla u, \nabla f \rp -|\Ric_-||\nabla u|^2.
\end{equation}
Put $$v=|\nabla u|^2 +\left\|
f^2\right\|^*_{p} . $$ Note that when $f$ is constant, one can iterate with $v=|\nabla u|^2$ and the proof simplifies.

 For any function $\eta\in C_0^\infty(B_1(x))$ and constant $q >1$, compute
\begin{eqnarray*}
	\int|\nabla (\eta v^{q/2})|^2&=&-\int\eta v^{q}\Delta\eta -2\int \eta v^{q/2} \lp \nabla \eta, \nabla v^{q/2} \rp -\int\eta^2 v^{q/2}\triangle v^{q/2}\nonumber\\
	& = & \int(2|\nabla \eta |^2 - \eta \Delta\eta)  v^{q} -2 \int v^{q/2} \lp \nabla \eta, \nabla (\eta v^{q/2}) \rp \\ & & - (1-\frac 2q)\int |\nabla (\eta v^{q/2}) - v^{g/2} \nabla \eta|^2 -\frac q2 \int\eta^2 v^{q-1}\triangle v.
\end{eqnarray*}

By regrouping,
 \begin{eqnarray*}
 	\int|\nabla (\eta v^{q/2})|^2 &  = & \frac{q}{2(q-1)}\int( (1+\frac 2q)|\nabla \eta |^2 - \eta \Delta\eta)  v^{q}  \\
 	& & -\frac{1}{q-1}\int v^{q/2} \lp \nabla \eta, \nabla (\eta v^{q/2}) \rp -\frac{q^2}{4(q-1)}\int\eta^2v^{q-1}\triangle v  \\
 	&  \leq &\frac{1}{2}\int|\nabla (\eta v^{q/2})|^2+\frac{1}{2}\frac{q^2+q-1}{(q-1)^2}\int|\nabla\eta|^2 v^q \\ & & - \frac{q}{2(q-1)}\int \eta  v^{q}\Delta\eta  -\frac{q^2}{4(q-1)}\int\eta^2v^{q-1}\triangle v  \nonumber.
 \end{eqnarray*}
 Hence,
\be  \label{grad-cut-comp}
	\int|\nabla (\eta v^{q/2})|^2 \le \frac{q^2+q-1}{(q-1)^2}\int|\nabla\eta|^2 v^q  - \frac{q}{q-1}\int \eta  v^{q}\Delta\eta  -\frac{q^2}{2(q-1)}\int\eta^2v^{q-1}\triangle v.
		\ee
Now plugging (\ref{Bochnerformula}) into (\ref{grad-cut-comp}), we have
\begin{eqnarray*}
	\int|\nabla (\eta v^{q/2})|^2
	&  \leq &\frac{q^2+q-1}{(q-1)^2}\int|\nabla\eta|^2 v^q - \frac{q}{q-1}\int \eta  v^{q}\Delta\eta \\ & &+ \frac{q^2}{q-1}\int\eta^2v^q|\Ric_-|
	  - \frac{q^2}{q-1}\int\eta^2 v^{q-1} \lp \nabla u, \nabla f \rp  \nonumber.
\end{eqnarray*}
For the last term, we have
\ban
\lefteqn{ \int\eta^2 v^{q-1} \lp \nabla u, \nabla f \rp} \\
 & = & - \int \eta^2 v^{q-1} f^2 - 2 \int \eta f  v^{q-1} \lp \nabla u, \nabla \eta \rp  - (q-1)  \int \eta^2 f v^{q-2} \lp \nabla u, \nabla v \rp  \\
& \ge &  - \int 6\eta^2 v^{q-1} f^2 -  \int |\nabla \eta|^2 v^q - \frac{2(q-1)}{q^2} \cdot \frac 18 \int\eta^2|\nabla v^{q/2}|^2 \\
& \ge & - \frac{2(q-1)}{q^2} \cdot \frac 14 \int|\nabla (\eta v^{q/2})|^2  - \int 6\eta^2 v^{q-1} f^2 - (1+ \frac{q-1}{2q^2}) \int |\nabla \eta|^2 v^q.
\ean
To control $\Delta \eta$, we choose a more specific cur-off function. For $0<r<1$,  let $\varphi \in C_0^\infty (R)$ be a cut-off function such that $0\le \varphi \le 1, \ \varphi(t) \equiv 1$ for $t \in [0,r]$, $\varphi(t) \equiv 0$ for $t \ge 1$,  and $\varphi' \le 0$. Then define
\be
\eta (y) = \varphi (d(x,y)),  \label{eta}
\ee
 where $d(x,y)$ is the distance function from $x$.
Thus  $|\nabla \eta |= |\varphi'|$, and
\ban
\Delta \eta  &= & \varphi'' + \varphi' \Delta d =   \varphi'' + \varphi' (\Delta d -\frac{n-1}{d} + \frac{n-1}{d}) \\
& \ge &  \varphi'' + \varphi' (\psi + \frac{n-1}{d}) \ge -|\varphi''| - \frac{|\varphi'|}{r} -|\varphi' | \psi,
\ean
where $\psi = (\Delta d -\frac{n-1}{d})_+$.

Therefore we have, for $ q \ge \frac{n}{n-2}$,
\ban
\lefteqn{ \int|\nabla (\eta v^{q/2})|^2 } \\
& \leq & C(n)q\int \left[ \left(|\varphi''| + \frac{|\varphi'|}{r} +|\varphi' | \psi \right) \eta v^q+ |\varphi'|^2 v^q  + \eta^2  f^2 v^{q-1} + \eta^2|\Ric_-|v^q\right].
\ean
Notice that this formula remains valid for $q=1$.  In fact
\[
|\nabla (\eta v^{1/2})|^2  = \left| v^{1/2} \nabla \eta + \eta \frac{|\nabla u |}{v^{1/2}} \nabla |\nabla u| \right|^2
 \le   2 v |\nabla \eta|^2 + 2 \eta^2 |\Hess u|^2,   \]
and
\begin{eqnarray}
\int\eta^2|\Hess u|^2&=&-\int\nabla_iu(2\eta\nabla_j\eta\nabla_i\nabla_ju+\eta^2\nabla_i\triangle u+\eta^2R_{ij}\nabla_ju)\nonumber\\
&\le&\frac{1}{2}\int\eta^2|\Hess u|^2+3\int|\nabla\eta|^2 v+\int \eta^2 f^2 + \int\eta^2|\Ric_-| v\nonumber.
\end{eqnarray}

Denote $\mu=\frac{n}{n-2}$.   Applying the Sobolev inequality (\ref{Sobolev}),   we obtain for $q\geq \frac{n}{n-2}$ and $q=1$,
\begin{eqnarray}  \label{main-ineq1}
 \\
 &&\bigg(\fint_{B_1(x)}(\eta^2 v^q)^\mu\bigg)^{1/\mu} \nonumber  \\
& \leq &  C^2_s(B_1(x)) C(n)q\fint_{B_1(x)} \left[  \left(|\varphi''| + \frac{|\varphi'|}{r} +|\varphi' | \psi \right) \eta v^q+ |\varphi'|^2 v^q+ f^2 \eta^2 v^{q-1} +|\Ric_-| \eta^2 v^q\right]. \nonumber
\end{eqnarray}
The integration involving Ricci curvature can be estimated as follows. For $p> \frac n2$,
\begin{eqnarray}
\fint_{B_1(x)}|Ric_-|\eta^2 v^q&\leq&\|\Ric_-\|_p^*\cdot\big(\fint_{B_1(x)}(\eta^2v^q)^{\frac{p}{p-1}}\big)^{\frac{p-1}{p}}\nonumber\\
&\leq& \|\Ric_-\|_p^* \bigg(\fint_{B_1(x)}\eta^2v^q\bigg)^{\frac{p-1}{p}a}\cdot\bigg(\fint_{B_1(x)}(\eta^2v^q)^\mu\bigg)^{(1-a)\frac{p-1}{p}}\nonumber\\
&\le& \|\Ric_-\|_p^* \left[ \epsilon\bigg(\fint_{B_1(x)}(\eta^2v^q)^\mu\bigg)^{\frac{1}{\mu}}+\epsilon^{-\frac{(1-a)\mu}{a}}
\cdot\bigg(\fint_{B_1(x)}\eta^2v^q\bigg)\right] \nonumber.
\end{eqnarray}
where $a=a(n,p)=\frac{2p-n}{2(p-1)}>0$ is determined via
$$a+(1-a)\mu=\frac{p}{p-1}.$$
Here we used Young's inequality
$$xy\leq\epsilon x^\gamma+\epsilon^{-\frac{\gamma^*}{\gamma}}y^{\gamma^*},\hspace{0.5cm}\forall x,y\geq 0,\gamma>1,(\gamma^*)^{-1}+\gamma^{-1}=1,$$
where
$$\gamma=\frac{p}{(1-a)(p-1)\mu},\, \ \gamma^*=\frac{p}{(p-1)a}.$$
By setting $\epsilon= \left( 4C(n)q C^2_s)  \|\Ric_-\|_p^*\right)^{-1}$,  we conclude
\ba  \label{Ric-term}
\lefteqn{C(n)q C^2_s  \fint_{B_1(x)}\eta^2|\Ric_-|v^q} \\
 &\le  \frac{1}{4}\bigg(\fint_{B_1(x)}(\eta^2v^q)^\mu\bigg)^{\frac{1}{\mu}}
+ C(n,p) \big(q C^2_s\|\Ric_-\|_p^*\big)^{\frac{2p}{2p-n}}
\cdot\bigg(\fint_{B_1(x)}\eta^2v^q\bigg). \nonumber
\ea

For the term $\fint_{B_1(x)} \eta^2 f^2 v^{q-1}$, since $v \ge \left\|
f^2\right\|^*_{p} $, we have
\[
\fint_{B_1(x)} \eta^2 f^2 v^{q-1} \le \frac{1}{\left\|f^2\right\|^*_{p}}\fint_{B_1(x)} \eta^2 f^2 v^{q} \le  \big(\fint_{B_1(x)}(\eta^2 v^q)^{\frac{p}{p-1}}\big)^{\frac{p-1}{p}}. \]
Now the same argument as above with $\epsilon=\left(4C(n)C^2_s)q \right)^{-1}$ gives
\ba  \label{f-term}
\lefteqn{ C(n)C^2_s q \fint_{B_1(x)}\eta^2 f^2 v^{q-1} }    \\
&\le & \frac{1}{4}\bigg(\fint_{B_1(x)}(\eta^2 v^{q})^\mu\bigg)^{\frac{1}{\mu}}
+  C(n,p)\big(C^2_s q\big)^{\frac{2p}{2p-n}}
\cdot\bigg(\fint_{B_1(x)}\eta^2 v^{q}\bigg). \nonumber
\ea
For the term with $\psi$, using the H\"older inequality and the Laplacian comparison estimate (\ref{Laplacian-com-average-norm}),
\ban
\fint_{B_1(x)} \psi \eta |\varphi'|  v^q&\leq&\|\psi\|_{2p}^*\cdot \|\eta  \varphi' v^q\|^*_{\frac{2p}{2p-1}} \\
& \le & C(n,p) \left(\|\Ric_-\|_p^*\right)^{1/2} \cdot \|\eta  \varphi' v^q\|^*_{\frac{2p}{2p-1}} .
\ean

Note that for $b=\frac{p(n-2)}{n(2p-1)}<1$,
\ban
\|\eta  \varphi' v^q\|^*_{\frac{2p}{2p-1}} & = & \left[ \fint_{B_1(x)} \left(\eta^2  v^q\right)^{b\mu} \left(|\varphi'|^2  v^q\right)^{\frac{p}{2p-1}} \right]^{\frac{2p-1}{2p}}\\
& \leq & \left[ \left(\fint_{B_1(x)} \left(\eta^2  v^q\right)^{\mu} \right)^b \left(\fint_{B_1(x)} \left(|\varphi'|^2  v^q\right)^{\frac{np}{np+2p-n}}\right)^{\frac{np+2p-n}{n(2p-1)}} \right]^{\frac{2p-1}{2p}}\\
& \leq & \left[ \left(\fint_{B_1(x)} \left(\eta^2  v^q\right)^{\mu} \right)^b \left(\fint_{B_1(x)} |\varphi'|^2  v^q\right)^{\frac{p}{2p-1}} \right]^{\frac{2p-1}{2p}} \\
& \le & \epsilon \left(\fint_{B_1(x)} \left(\eta^2  v^q\right)^{\mu} \right)^{1/\mu} + \frac{1}{4\epsilon}\fint_{B_1(x)} |\varphi'|^2  v^q.
\ean
Here we used the fact that $\frac{np}{np+2p-n}<1$ since $p>n/2$.

Choose $\epsilon = (4C(n)C_s^2q C(n,p) \left(\|\Ric_-\|_p^*\right)^{1/2})^{-1}$, we have
\ba \lefteqn{ C(n)C^2_s q \fint_{B_1(x)}\psi \eta |\varphi'|  v^q}  \label{psi-term}  \\
& \le & \frac{1}{4}\bigg(\fint_{B_1(x)}(\eta^2 v^{q})^\mu\bigg)^{\frac{1}{\mu}}
+ (q C(n)C_s^2)^2 C^2(n,p)\|\Ric_-\|_p^* \fint_{B_1(x)} |\varphi'|^2  v^q.  \nonumber
\ea

Plugging the three estimates (\ref{Ric-term}), (\ref{f-term}), (\ref{psi-term}) into the inequality (\ref{main-ineq1}) gives
\begin{eqnarray*}
\lefteqn{\bigg(\fint_{B_1(x)}(\eta^2 v^q)^\mu\bigg)^{1/\mu}} \\
 & \leq& 4 C^2_s C(n)q \left[\fint_{B_1(x)}   \left(|\varphi''| + \frac{|\varphi'|}{r}  \right) \eta v^q+ \left(1+q C_s^2 C^2(n,p)\|\Ric_-\|_p^*\right) \fint |\varphi'|^2 v^q  \right]  \\
& &
+C(n,p)\big( C^2_s q\big)^{\frac{2p}{2p-n}}\left(1+( \|\Ric_-\|_p^*)^{\frac{2p}{2p-n}} \right) \bigg(\fint_{B_1(x)}\eta^2v^q\bigg)\nonumber.
\end{eqnarray*}

Define $q_k=\mu^k$, $k\geq 0$, and $r_{k}=(\frac{3}{4}-\sum_{i=0}^k2^{-i-3})$. Choose cut-off functions $\eta_k=\varphi_k\circ d\in C_0^\infty(B_{r_k}(x))$ such that
$$\eta_k\equiv 1,\hspace{0.3cm}\mbox{ on }B_{r_{k+1}}(x);\hspace{0.3cm}|\varphi'_k|\leq2^{k+5}, \hspace{0.3cm}|\varphi''_k|\leq2^{2k+10}.$$
Then substituting $\eta_k$ into the estimate and running the iteration for any $k\geq 0$ we get
$$\|v\|^*_{\infty,B_{\frac{1}{2}}(x)}\leq C(n,p)  A^{n/2}   \|v\|^*_{1,B_{\frac{3}{4}}(x)},$$
where
$$ A= C^2_s (1+ C_s^2\|\Ric_-\|_p^*) +  C_s^{\frac{4p}{2p-n}} \left(1+( \|\Ric_-\|_p^*)^{\frac{2p}{2p-n}} \right) .
$$

Finally observe that, by integrating by parts, for $\eta\in C^\infty_0(B_{1}(x))$ with $\eta\equiv 1$ in $B_{\frac{3}{4}}(x)$, and $|\nabla \eta | \le 5$, we have
\ban
\fint_{B_{\frac{3}{4}}(x)}|v| & \leq &  \frac{\vol B_1(x)}{\vol B_{\frac 34}(x)}  \fint_{B_{1}(x)} \eta^2(|\nabla u|^2 +\left\|
f^2\right\|^*_{p})  \\
& \leq &    \frac{\vol B_1(x)}{\vol B_{\frac 34}(x)}  \left[  \left\|
f^2\right\|^*_{p} + \fint_{B_{1}(x)} \eta^2 (u^2 +f^2) + 8 \fint_{B_{1}(x)}|\nabla\eta|^2u^2  \right] \\
&  \leq  & 201 \cdot \frac{\vol B_1(x)}{\vol B_{\frac 34}(x)}  \cdot  (\|u\|_2^* )^2 + 2(\|f\|_{2p}^*)^2.
\ean
This gives the gradient estimate.
\end{proof}

Combining (\ref{volume doubling}) and (\ref{L2-Sobolev-r}) with Theorem~\ref{gradient-g}  gives Theorem~\ref{gradient}.

Later on we will need Harnack inequality for harmonic function. Hence we also give a gradient estimate for $\ln u$ as in Cheng-Yau's gradient estimate \cite{cheng-yau}. In the proof we need Li-Schoen's trick of bounding high power by lower power \cite{Li-Schoen84}.
\begin{theorem}
Assume as in above theorem. Let $u$ be a positive harmonic function in $B_R(x)$, then
\begin{equation*}
\sup_{B_{\frac{R}{2}}(x)}|\nabla \ln u|^2 \leq C\left(n,p, R^{-2} C_s^2, k(p,R) \right)  R^{-2}  \frac{\vol B_R(x)}{\vol B_{\frac 45 R}(x)} .
\end{equation*}
\end{theorem}
\begin{proof}
By scaling we may assume $R=1$.  Let $h = \ln u, v=|\nabla h|^2$. Then $\Delta h = -v$.  From the Bochner formula,
\begin{eqnarray*} \label{Bochnerformula2}
\frac{1}{2}\triangle|\nabla h|^2 & = & |\Hess h|^2+ \lp \nabla h, \nabla \Delta h \rp + \Ric(\nabla h,\nabla h)  \\
& \ge &  \frac{v^2}{n} - \lp \nabla h, \nabla v \rp - |\Ric_-| v.
\end{eqnarray*}
For any $\eta\in C_0^\infty(B_1(x))$, $l \ge 0$, multiply above by $v^l \eta^2$ and integrate on $B_1(x)$ gives,
\be  \label{integral-Bochner-ineq}
\frac{1}{2} \int v^l \eta^2 \triangle v \ge   \int \frac{ \eta^2 v^{l+2}}{n} - \int v^l \eta^2 \lp \nabla h, \nabla v \rp - \int v^{l+1} \eta^2 |\Ric_-|.
\ee
We compute
\[ \int v^l \eta^2 \lp \nabla h, \nabla v \rp = - \int v^{l+1}\eta^2 \Delta h - l \int v^l \eta^2 \lp \nabla h, \nabla v \rp - 2   \int v^{l+1}\eta \lp \nabla h, \nabla \eta \rp.\]
Hence
\ba
 \int v^l \eta^2 \lp \nabla h, \nabla v \rp  & =  & \frac{1}{l+1} \int  v^{l+2}\eta^2 - \frac{2}{l+1} \int v^{l+1}\eta \lp \nabla h, \nabla \eta \rp \nonumber \\
& \le &  \frac{2}{l+1} \int  v^{l+2}\eta^2 + \frac{1}{l+1} \int v^{l+1}  | \nabla \eta |^2.   \label{right-second-term}
\ea
By (\ref{grad-cut-comp}),
\ba
\lefteqn{\frac{(l+1)^2}{2l}\int v^l \eta^2 \triangle v} \nonumber \\
& & \le  -\int  \left| \nabla \left( \eta v^{\frac{l+1}{2}} \right) \right|^2 + \frac{(l+1)^2 +l}{l^2}  \int  v^{l+1} | \nabla \eta|^2 - \frac{l+1}{l}  \int \eta v^{l+1} \Delta \eta  \label{left-term}
\ea
Plugging (\ref{left-term}) and (\ref{right-second-term}) into (\ref{integral-Bochner-ineq}), we have
\ban
 -\int  \left| \nabla \left( \eta v^{\frac{l+1}{2}} \right) \right|^2 & \ge & \frac{(l+1)^2}{l} \left(\frac 1n - \frac{2}{l+1}\right)  \int  v^{l+2}\eta^2  + \frac{l+1}{l}  \int \eta v^{l+1} \Delta \eta \\
 & & - \frac{2l^2 +4l +1}{l^2}\int v^{l+1}  | \nabla \eta |^2  - \frac{(l+1)^2}{l} \int v^{l+1} \eta^2 |\Ric_-|.
 \ean
When  $l \ge 2n -1$, choose $\eta$ as in (\ref{eta}),  we have
\[
 \int  \left| \nabla \left( \eta v^{\frac{l+1}{2}} \right) \right|^2  \le  C(n) l \int \left[ \left(|\varphi''| + \frac{|\varphi'|}{r} +|\varphi' | \psi \right) \eta v^{l+1}  +  v^{l+1}  |\varphi' |^2  +  v^{l+1} \eta^2 |\Ric_-| \right]
\]

Use Sobolev inequality (\ref{Sobolev}) and estimate as in (\ref{Ric-term}), (\ref{psi-term}) and iterate from $l = 2n-1$ as in Theorem~\ref{gradient-g}, we have
\be
\|v\|^*_{\infty,B_{\frac{1}{2}}(x)}\leq C(n,p)  A^{n/2}   \|v\|^*_{2n-1,B_{\frac{3}{4}}(x)},  \label{grad-2n-1}
\ee
where
$$ A= C^2_s (1+ C_s^2\, \|\Ric_-\|_p^*) +  C_s^{\frac{4p}{2p-n}} \left(1+( \|\Ric_-\|_p^*)^{\frac{2p}{2p-n}} \right) .
$$

Since we have volume doubling, by the proof of Theorem 2.1 in \cite{Li-Schoen84}, we can lower the power $2n-1$ in (\ref{grad-2n-1}) by adjusting the size of the balls. Namely we have
\[
\|v\|^*_{\infty,B_{\frac{1}{4}}(x)}\leq C\left(n,p, C_s^2, \|\Ric_-\|_p^*\right) \,    \|v\|^*_{1,B_{\frac 45}(x)}.
\]
For the $L^1$ bound, since $v = - \Delta h$,
\[
\int_{B_1(x)} \eta^2 v = - \int_{B_1(x)} \eta^2 \Delta h = 2  \int_{B_1(x)}   \eta \lp \nabla \eta, \nabla h \rp  \le \frac 12  \int_{B_1(x)} \eta^2 v + 2 \int_{B_1(x)} | \nabla \eta |^2, \]
where $\eta \in C_0^\infty (B_1(x))$ is a cut-off function with $\eta = 1$ on $B_{\frac 45}(x)$ and $|\nabla \eta| \le 6$. 

Hence \[  \|v\|^*_{1,B_{\frac 45}(x)} \le 144 \,  \frac{\vol B_1(x)}{\vol B_{\frac 45}(x)}.\]
\end{proof}

With Theorem~\ref{gradient} one can prove as in \cite[Theorem 6.4]{PeWe00} the following.

\begin{lemma} \label{cut-off}
For any integer $n$ and $p>\frac{n}{2}$ there exist $\varepsilon$ and $C$ such that the following holds. Let $M$ be a complete $n$-dimensional Riemannian manifold satisfying $k(p,1)\le\varepsilon$. For any metric ball $B_r(x)$ with $\partial B_r(x)\neq\emptyset$, $r\le 1$, there exists $\phi\in C^\infty_0(B_r(x))$ satisfying
\begin{equation*}
0\le\phi\le 1,\, |\nabla\phi|^2+|\triangle \phi|\le Cr^{-2}.
\end{equation*}
\end{lemma}

With the (relative) local Sobolev constant estimate (\ref{L1-local-sobolev}), one gets heat kernel upper bound, see e.g. \cite[(2.17)]{Hesisch-Saloff-Coste2001}. With this and the volume doubling (\ref{volume doubling}), Zhang-Zhu obtained Li-Yau's gradient estimate
\cite{Zhang-Zhu2016}. Hence one has
parabolic Harnack inequality.  With this we have the local heat kernel lower bounds as in \cite[Lemma 2.3]{CN12}.  Namely,  we have
\begin{theorem} \label{heat-bounds}
Let $M$ be an $n$-dimensional Riemannian manifold, and $p>n/2$. There is an $\varepsilon \left(
n,p\right) >0$ and  $C\left( n,p\right) >1$ such that if   $k(p,1) \leq \varepsilon$, then for any real number $s$,  $0<r<1$, $x \in M$  and nonnegative solution $u$ of the heat equation in $Q = (s-r^2,s) \times B_r(x)$,
\[
\sup_{Q_-}  u  \le C \inf_{Q_+}  u,
\]
where $Q_- = (s-\frac 34 r^2, s-\frac 12 r^2) \times B_{\frac 12 r} (x)$,  $Q_+ =(s-\frac 14 r^2 , s) \times B_{\frac 12 r} (x).$

The heat kernel $H(x,y,t)$ satisfies the two-sided Gaussian bound
\[
\frac{c_1}{\vol B_{\sqrt{t}}(x)} e^{-\frac{d^2(x,y)}{c_2t}} \le H(x,y,t) \le \frac{C_1}{\vol B_{\sqrt{t}}(x)} e^{-\frac{d^2(x,y)}{C_2t}}
\]
for all $t \in (0, 1)$ and $x, y \in M$.
\end{theorem}

For our purpose, we need two-sided bound on the Dirichlet heat kernels of the balls. Let $H^B(x,y,t)$ be the Dirichlet heat kernel of the ball $B_\rho (x)$  with $t \in (0, 1)$ and $\rho \ge \sqrt{t}$. Note that  $H^B(x,y,t) \le H(x,y,t)$. With the local volume doubling and Poincare inequality, by  \cite[(3.4)]{Hesisch-Saloff-Coste2001}, there exist constants (depending only on the constants from the volume doubling and Poincare inequality) $a$, $\tau$ (small),  $A$(large) and $c$ such that \[
H^B(x,y,t) \le \frac{c}{\vol B_{\sqrt{t}}(x) }
\]
for $y \in B_{a\sqrt{t}}(x)$, $t \in (0, \tau)$,  $\rho \ge A \sqrt{t}$. By making $\varepsilon$ smaller (and a rescaling argument as at the end of the proof of Theorem \ref{local-iso}) we obtain

\begin{theorem}[Dirichlet Heat kernel upper and lower bounds]
For any integer $n$ and $p>\frac{n}{2}$ there exist $\varepsilon$ and $C$ such that the following holds. Let $M$ be a complete $n$-dimensional Riemannian manifold satisfying $k(p,1)\le\varepsilon$. Let $H^{B_r(x)}(x,y,t)$ be the Dirichlet heat kernel of the ball $B_r(x)$.  Then
\be
H^{B_r(x)}(x,y,t) \le \frac{C}{\vol B_{\sqrt{t}}(x)} e^{-\frac{d^2(x,y)}{5t}},\ \forall x, y \in M \ \mbox{with} \  0<t \le 1
\ee
and
\be
H^{B_r(x)}(x, y, t) \ge     \frac{C^{-1}}{\vol B_{\sqrt{t}}(x)}, \ 0<t \le \frac{1}{2}  r^2, \ y \in B_{10\sqrt{t}}(x).
\ee
\end{theorem}

This Dirichlet  heat kernel upper and lower bounds give the quantitive mean value inequality.

\begin{proposition}  \label{p-mean} Under the assumption above, let $u$ be a nonnegative function satisfying
\[
\frac{\partial}{\partial t} u \ge \Delta u - f,\]
where $f$ is a nonnegative space-time function.
Then,  for $q > \frac n2$,
\be
\fint_{B_{\frac{1}{2} r} (x)} u(\cdot, 0) d \vol  \le C u(x, r^2) + C(n,p,q)\,  r^2 \sup_{t \in [0, r^2]}  \|f(t)\|^*_{q, B_r(x)}
\ee
for all $x \in M$, $r \le 1 $.
\end{proposition}
\begin{remark} For our application it's crucial that the norm of $f$ is a normalized local norm instead of the global norm in \cite{TiZh13}. It recovers Lemma 2.1 in \cite{CN12}, where it is proven when $f$ is constant. The key here is to use Dirichlet heat kernel of balls.
	\end{remark}
\begin{proof}
Compute \ba  \label{Duhamel}
\lefteqn{ \frac{d}{dt} \int_{B_r(x)} u(y,t) H^{B_r(x)}(x, y, r^2-t) d \vol (y) } \\
&  = & \int_{B_r(x)} \left[ H^{B_r(x)}(x,y,r^2-t) \left( \frac{\partial}{\partial t} - \Delta \right) u(y,t) \right] d\vol (y)  \nonumber \\
& \ge &  - \int_{B_r(x)} \left[ H^{B_r(x)}(x,y,r^2-t) f(y,t)  \right] d\vol (y)  \nonumber
\ea
By the upper bound of $ H^{B_r(x)}(x,y, r^2-t)$, we have, for $q>1$,
\ban \lefteqn{ \int_{B_r(x)} \left[ H^{B_r(x)}(x,y,r^2-t) f(y,t)  \right] d\vol (y)} \\
 & \le & \|f(t)\|_{q, B_r(x)} \,  \|H^{B_r(x)}(x,y,r^2-t)\|_{\frac{q}{q-1}, B_r(x)}  \\
 & \le &   \frac{C}{\vol B_{\sqrt{r^2-t}}(x)} \|f(t)\|_{q, B_r(x)} \left( \int_{B_r(x)}  e^{- \frac{q}{q-1}\frac{d^2(x,y)}{5(r^2-t)}}
 d\vol (y) \right)^{1-\frac 1q} \\
 & \le & C \left(\vol B_{\sqrt{r^2-t}}(x) \right)^{-\frac 1q} \|f(t)\|_{q, B_r(x)} \\
 & = & C \left( \frac{\vol B_r(x)}{\vol B_{\sqrt{r^2-t}}(x)}\right)^{\frac 1q}\|f(t)\|^*_{q, B_r(x)} \\
 & \le & C r^{\frac nq} (r^2-t)^{-\frac{n}{2q}}  \|f(t)\|^*_{q, B_r(x)}.
 \ean
 In the last step we used the  volume doubling property  (\ref{volume doubling}).

 Now integrate (\ref{Duhamel}) from $0$ to $r^2$ gives
 \ban
 u(x, r^2) & \ge &  \int_{B_r(x)} u(y,0) H^{B_r(x)}(x, y, r^2) d \vol (y) - Cr^{\frac nq} \int_0^{r^2} (r^2-t)^{-\frac{n}{2q}}  \|f(t)\|^*_{q, B_r(x)} dt  \\
 & \ge & C^{-1} \fint_{B_{\frac{1}{2} r}(x)} u(y,0) d\vol (y) - C(n,p,q) r^2 \sup_{t \in [0, r^2]}  \|f(t)\|^*_{q, B_r(x)}.
 \ean
 Here we used the lower bound for $ H^{B_r(x)}(x, y, r^2)$ on $B_{\frac{1}{2} r}(x)$ and $q > \frac n2$.
\end{proof}

\begin{corollary} \label{Lap-mean}
	Assume as above.  Let $u$ be a nonnegative function satisfying
	\[ \Delta u \le f. \]
	Then,  for $q > \frac n2$,
	\be
	\fint_{B_{\frac{1}{2} r} (x)} u\, d \vol  \le C\left( u(x) + r^2 \|f_+\|^*_{q, B_r(x)} \right)
	\ee
	for all $x \in M$, $r \le 1 $.
\end{corollary}
This is the $L^1$ Harnack inequality. For the Euclidean case, see e.g. \cite[Theorem 4.15]{Han-Li-book}. We would like to thank Ruobin Zhang for this reference.

With the above tools, we can extend Colding-Naber's
$L^2$ Hessian estimate for the parabolic approximation of the distance function to the integral curvature setting without essentail difficuties.  In the noncollapsed case it is done in \cite{TiZh13}, see also \cite{Zhang-Zhu15}.

In what follows, we always assume $p>\frac{n}{2}$ and  $M$ is  a complete $n$-dimensional Riemannian manifold satisfying $k(p,1)\le\varepsilon(n,p)$  for the $\varepsilon$ so the results above all hold.

Fix two points $y_-, y_+$ in $M^n$, the excess is \[
e(x) = d(y_-, x) + d(y_+,x) - d(y_-, y_+). \]
Define
\[
b_+(x) = d(y_+, x) -  d(y_-, y_+), \ \ b_-(x)  = d(y_-, x) . \] Hence $e(x) = b_+(x) +  b_-(x) $.
Note that \[
\Delta b_{\pm} (x)  \le \frac{n-1}{d(x, y_{\pm})} + \psi_{\pm} ,  \] where $\psi_{\pm} =\left(  \Delta d(y_{\pm}, x)  -  \frac{n-1}{d(x, y_{\pm})} \right)_+$ is the error term of the Laplacian comparison.

Denote $d_0 = d(y_-, y_+)$. Without loss of generality, assume $d_0 \le 1$. Denote by $A_{r_1, r_2} = A_{r_1d_0, r_2d_0}(\{y_-, y_+\})$ the annulus for the set $\{y_-, y_+\}$, with $0 <r_1 <r_2 $.
Then Corollary~\ref{Lap-mean} and the Laplacian comparison estimate (\ref{Laplacian-com-average-norm}) gives
\begin{theorem}
 Fix some small postive constant  $\delta >0$.   There exist $\bar{\epsilon} = \bar{\epsilon}(n,p,\delta)$  and $C = C(n,p,\delta)$ such that  for all $0 < \epsilon < \bar{\epsilon} $, $x \in A_{\frac{\delta}{4}, 16}$,
\[
\fint_{B_{\epsilon d_0}(x) } e(y) dy \le C \left[ e(x) + (\epsilon d_0)^2 ( \|\psi_-\|^*_{2p, d_0}  +  \|\psi_+\|^*_{2p, d_0} ) \right] \le C \left[ e(x) + \epsilon^2  d_0  \right].
\]
In particular, if $e(x) \le \epsilon^2 d_0$, then
\[
e(y)  \le C \epsilon^{1+ \frac {1}{n+1}} d_0, \ \forall y  \in B_{\frac 12 \epsilon d_0}(x) . \]
\end{theorem}
\begin{remark}
We obtain the optimal integral bound for the excess as in the pointwise Ricci lower bound case \cite[Theorem 2.6]{CN12}, compare \cite[Corollary 2.19]{TiZh13}, \cite[Lemma 4.9]{Zhang-Zhu15}. For the pointwise estimate, note that Abresch-Gromoll's original estimate gives $\epsilon^{1+ \frac{1}{n-1}}$ \cite{Abresch-Gromoll1990}.
\end{remark}

As in \cite{CN12}, one can extend Lemma~\ref{cut-off} to annulus so we have the cut-off function $\phi$ such that \[
\phi = 1  \ \mbox{in} \ A_{\frac{\delta}{4}, 8};  \ \phi = 0 \ \mbox{outside} \ A_{\frac{\delta}{16}, 16}  \]
and \[
|\nabla \phi|^2 + |\Delta \phi|  \le C(n,p, \delta). \]
Define the parabolic approximation functions ${\bf b}_{\pm, t}$ and ${\bf e}_t$ by
\[
{\bf b}_{\pm, t} (x)= \int H(x,y,t) \phi(y) b_{\pm} (y) dvol(y) \]
and
\[
{\bf e}_t  (x) =  \int H(x,y,t) \phi(y) e (y) dvol(y).  \] Then\[
{\bf e}_t   = {\bf b}_{+, t} + {\bf b}_{-, t}. \]
Following \cite{TiZh13}, we have following estimates for the approximates, which play important role in the Cheeger-Colding-Naber local theory for Gromov-Hausdorff limits.

\begin{theorem}
There exists $C =C(n,p, \delta)$ such that for all $0 < \epsilon \le \bar{\epsilon}(n,p,\delta)$,   any $x \in A_{\frac{\delta}{2}, 4}$ with $e(x) \le \epsilon^2 d_0$ and any $\epsilon$-geodesic $\sigma$ connecting $y_-, y_+$, there exists $r \in[\frac 12,2]$ with

1. $\left|{\bf b}_{\pm, r\epsilon^2 d_0^2} - b_{\pm} \right| \le  C d_0 ( \epsilon^2+ \epsilon^{2-\frac{n}{2p}})$.

2. $\fint_{B_{\epsilon d_0}(x)} \left| | \nabla {\bf b}_{\pm, r\epsilon^2 d_0^2} |^2 - 1 \right| \le C (\epsilon + \epsilon^{1-\frac{n}{2p}})$.

3. $\fint_{\delta d_0}^{(1-\delta)d_0}  \fint_{B_{\epsilon d_0}(\sigma(s))} \left| | \nabla {\bf b}_{\pm, r\epsilon^2 d_0^2} |^2 - 1 \right| \le C( \epsilon^2+ \epsilon^{2-\frac{n}{p}})$.

4. $\fint_{\delta d_0}^{(1-\delta)d_0}  \fint_{B_{\epsilon d_0}(\sigma(s))} \left| \Hess {\bf b}_{\pm, r\epsilon^2 d_0^2} \right|^2 \le \frac{C(1+\epsilon^{-\frac{n}{p}})}{d_0^2}.$
\end{theorem}

We will only show the first lemma here to indicate the difference.
\begin{lemma}
There exists a constant $C = C(n,p,\delta)$ such that
\[
\Delta {\bf b}_{\pm, t}, \Delta  {\bf e}_t  \le C \left( \frac{1}{d_0} +  t^{-\frac{n}{4p}} \right)
\]
for $t < 1$.
\end{lemma}
\begin{proof}
Since, for $x \in A_{\frac{\delta}{16}, 16}$,
\[
\Delta(\phi b_+) = b_+  \Delta \phi + 2 \lp \nabla \phi, \nabla b_+ \rp + \phi \Delta b_+ \le C d_0^{-1} + \psi_+,\]
we have
\ban
\Delta {\bf b}_{+, t} (x) & = & \int_{ A_{\frac{\delta}{16}, 16}}  \Delta_x H(x,y,t) \phi(y) b_{+}(y)  dvol(y)  \\
& = &  \int_{ A_{\frac{\delta}{16}, 16}}  \Delta_y H(x,y,t) \phi(y) b_{+}(y)  dvol(y)  \\
& = &  \int_{ A_{\frac{\delta}{16}, 16}}  H(x,y,t)  \Delta_y( \phi(y)  b_{+}(y) )  dvol(y)  \\
& \le & \frac{C}{d_0} + \int_{ A_{\frac{\delta}{16}, 16}}  H(x,y,t) \psi_+  dvol(y) .
\ean
Using the upper bound for $H(x,y,t)$ and argue as in Proposition~\ref{p-mean}, we have
\ban
 \int_{ A_{\frac{\delta}{16}, 16}}  H(x,y,t) \psi_+  dvol(y)  &  \le &  \|\psi_+\|_{2p} \| H(x,y,t)\|_{\frac{2p}{2p-1}} \\
& \le & C(n,p) \|\psi_+\|_{2p}^*   t^{-\frac{n}{4p}}.
\ean
These give the estimate; the other terms are exactly the same.

\end{proof}

\section{Necessity of smallness of integral Ricci}
\label{smallness}

By exploring Yang's counter-exmaple \cite{Ya92}, we point out that the smallness of integral Ricci curvature, (\ref{Lp Ricci: average}), is a critical condition in order to get the $L^p$ version of Cheeger-Colding theory.

For any $k> 1$, let $M=(-1,1)\times T^{n-1}$ be a portion of a complete manifold with a family of warped product metric
\begin{equation}
g_\epsilon=dr^2+(\epsilon^2+r^2)^{k}g_F
\end{equation}
where $T$ is a compact torus with flat metric $g_F$ and $\epsilon>0$ is the parameter. A direct calculation gives the sectional curvature
\begin{align*}
  K(\frac{\partial}{\partial x^i},\frac{\partial}{\partial x^j}) & = -k^2r^2(\epsilon^2+r^2)^{-2}, \\
  K(\frac{\partial}{\partial x^i},\frac{\partial}{\partial r}) & = -k(\epsilon^2+r^2)^{-1}-k(k-2)r^2(\epsilon^2+r^2)^{-2},
\end{align*}
and the Ricci curvature
\begin{align*}
  Ric(\frac{\partial}{\partial x^i},\frac{\partial}{\partial x^i}) & =-\big[(n-2)k^2r^2
  +k(k-2)r^2+k(\epsilon^2+r^2)\big](\epsilon^2+r^2)^{-2}, \\
  Ric(\frac{\partial}{\partial r},\frac{\partial}{\partial r}) & = -(n-1)\big[k(k-2)r^2+k(\epsilon^2+r^2)\big](\epsilon^2+r^2)^{-2},\\
  Ric(\frac{\partial}{\partial x^i},\frac{\partial}{\partial r}) & = 0,
\end{align*}
where $(x^1,\cdots,x^{n-1})$ is the local normal coordinate of $T^{n-1}$.

Fix one $p\in(\frac{n}{2},\frac{k(n-1)+1}{2})$. In the following calculation, $\approx$ means equivalence up to a multiplication by a constant depending only on $n$ and $p$. The first observation is
\begin{equation}
|Rm_{g_\epsilon}| \approx |Ric_{g_\epsilon}| \approx k^2r^2(\epsilon^2+r^2)^{-2}+k(\epsilon^2+r^2)^{-1}.
\end{equation}
Put
\begin{equation}
B^*_r=\{(t,x)\in M|-r<t<r\}.
\end{equation}
Then, for any function $f$,
$$\int_{B_r^*}fdv_{g_\epsilon}=\int_{-r}^r\int_Tf(x,t)(\epsilon^2+t^2)^{\frac{(n-1)k}{2}}dv_{g_F}dt.$$
Applying to the curvature function we have, whenever $\epsilon<<r$,
\begin{eqnarray}
\int_{B_r^*}|Rm_{g_\epsilon}|^p&\approx& k^{2p}\int_0^r\int_T(\epsilon^2+t^2)^{\frac{(n-1)k}{2}-p}dv_{g_F}dt\nonumber\\
&\approx& \frac{k^{2p}}{k(n-1)+1-2p}\vol(T)r^{k(n-1)+1-2p}\nonumber\\
&\approx& k^{2p-1}\vol(T)r^{k(n-1)+1-2p},
\end{eqnarray}
whenever $k>>1$ is sufficiently large. Then notice that
$$\vol(B_r^*)\approx \frac{1}{k(n-1)+1}\vol(T)r^{k(n-1)+1}\approx k^{-1}\vol(T)r^{k(n-1)+1}.$$
Thus we have, whenever $\epsilon<<r$,
\begin{equation}
r^2\bigg(\fint_{B_r^*}|Rm_{g_\epsilon}|^p\bigg)^{\frac{1}{2p}}\approx k^2.
\end{equation}
In particular, the formula remains hold on the limit space,
\begin{equation}
r^2\bigg(\fint_{B_r(o)}|Rm_{g_0}|^p\bigg)^{\frac{1}{2p}}\approx k^2,\,\forall r>0.
\end{equation}

The family of manifolds $(M,g_\epsilon)$ has polynomial volume growth. Hence a generalized local Sobolev constant and Poincar\'e constant are still bounded\footnote{It should be understood as the $[k(n-1)+1]$-Sobolev inequality, but not the $n$-Sobolev inequality we have proved. Actually, the calculation shows the n-Sobolev constant is not bounded in Yang's example.}. However, since the limit collapse at the origin, the volume comparison of geodesic spheres fails. Furthermore, to guarantee the splitting property on the tangent cones of the limit space we eventually need that
$$r^2\bigg(\fint_{B_r(o)}|Ric_{g_0}|^p\bigg)^{\frac{1}{2p}}\rightarrow 0$$
as $r\rightarrow 0$, which never hold. Therefore, the boundedness of $L^p$ norm of curvature tensor is not sufficient to extend Cheeger-Colding theory.


\begin{thebibliography}{99}
\bibitem{Abresch-Gromoll1990}U. Abresch and D. Gromoll,
{\em On complete manifolds with nonnegative Ricci curvature,}
J. Amer. Math. Soc. 3 (1990), no. 2, 355 - 374.

\bibitem{An92} M. T. Anderson, {\it The $L^2$ structure of moduli spaces of Einstein metrics on 4-manifolds}, GAFA., 2 (1992), 29-89.

\bibitem{Aubry} E. Aubry, {\it Finiteness of $\pi_1$ and geometric inequalities in almost positive Ricci curvature}, Ann. Sci. \'{E}cole Norm. Sup. (4) 40 (2007), no. 4, 675 - 695.

\bibitem{Bamler} R. H. Bamler, {\it Compactness properties of Ricci flows with bounded scalar curvature},  arXiv:1512.08527.

\bibitem{Bamler-Zhang} R. H. Bamler and Q. S. Zhang, {\it Heat kernel and curvature bounds in Ricci flows with bounded scalar curvature},  Adv. Math.  319  (2017), 396 - 450. 

\bibitem{Buser1982} P. Buser, {\it A note on the isoperimetric constant},   Ann. Sci. \'Ecole Norm. Sup. (4)  15  (1982), no. 2, 213-230.

\bibitem{Chavel} I. Chavel, Riemannian Geometry: A Morden Introduction, Cambridge studies in advanced mathematics 98, Cambridge University Press, 2006.


\bibitem{Cheeger1970} J. Cheeger, {\em A lower bound for the
smallest eigenvalue of the Laplacian,}
in {\it Problems in Analysis}, R.~Gunning ed., 195-199. Princeton University Press,
1970.

%

\bibitem{Cheeger-Naber} J. Cheeger and A. Naber, {\it Regularity of Einstein manifolds and the codimension 4 conjecture},  Ann. of Math. (2) 182 (2015), no. 3, 1093-1165. 


\bibitem{cheng-yau}  S. Y. Cheng \& S. T. Yau, \emph{Differential equations
on Riemannian manifolds and their geometric applications}, Comm. Pure Appl.
Math. 28 (1975), 333-354.

\bibitem{CN12} T. Colding and A. Naber, {\it Sharp H\"older continuity of tangent cones for spaces with a lower Ricci curvature bound and applications}, Annals Math., 176 (2012), 1173-1229.

\bibitem{Federer-Fleming1960} H. Federer and W.H. Fleming,  {\it Normal integral currents}, Ann. Math 72 (1960), 458-520.

\bibitem{Ga88} S. Gallot, {\it Isoperimetric inequalities based on integral norms of Ricci curvature}, Soc. Math. de France, Asterisque 157-158 (1988), 191-216.


\bibitem{Gr} M. Gromov, {\it Paul Levy's isoperimetric inequality}, Appendix C  in  Metric Structures for Riemannian
and Non-Riemannian Spaces,   Progress in
Mathematics 152, Birkhauser, 2001.

\bibitem{Han-Li-book} Q. Han and F. Li, {Elliptic Partial Differential Equation}, Courant Inst. Math. Sci.; Amer. Math.
Soc., 1997.

\bibitem{Hesisch-Saloff-Coste2001} W. Hebisch and L. Saloff-Coste, {\it On the relation between elliptic and parabolic Harnack inequalities}, Ann. Inst. Fourier, Grenoble 51, 5(2001), 1437-1481.

\bibitem{Jiang-Naber} W. Jiang and A. Naber, {\em $L^2$ Curvature Bounds on Manifolds with Bounded Ricci Curvature},  arXiv:1605.05583. 

\bibitem{Li} P. Li, Geometric Analysis, Cambridge studies in advanced mathematics 134, Cambridge University Press, 2012.
\bibitem{Li-Schoen84} P. Li and R. Schoen, {\em $L^p$ and mean value properties of subharmonic functions on Riemannian manifolds,}. Acta Math. 153 (1984), no. 3-4, 279-301.

\bibitem{LiYa} F.H. Lin and X.P. Yang, Geometric Measure Theorey-An Introduction, Science Press, Beijing/New York and Intern. Press, Boston, 2002.

\bibitem{Paeng2011} S.-H. Paeng, {\it Buser's isoperimetric inequalities with integral norms of Ricci curvature},   Proc. Amer. Math. Soc.  139  (2011),  no. 8, 2903-2910.

\bibitem{Petersen&Sprouse}  P. Petersen and C. Sprouse, \emph{Integral
curvature bounds, distance estimates and applications}, J. Diff. Geo. 50
(1998) 269-298.

\bibitem{PeWe97}  P. Petersen and G.F. Wei, {\it Relative volume comparison with integral curvature bounds}, GAFA, 7 (1997), 1031-1045.

\bibitem{Rose2016} C.  Rose, {\em  Heat kernel upper bound on Riemannian manifolds with locally uniform Ricci curvature integral bounds},  J. Geom. Anal. 27 (2017), no. 2, 1737-1750. 


\bibitem{Simon} M. Simon, {\it Some integral curvature estimates for the Ricci flow in four dimensions}, arXiv:1504.02623

\bibitem{PeWe00} P. Petersen and G.F. Wei, {\it Analysis and geometry on manifolds with integral Ricci curvature bounds. II}, Trans. AMS., 353 (2000), 457-478.


\bibitem{Wei-Ye}G. Wei  and  R. Ye, {\em A Neumann Type Maximum Principle for the Laplace Operator on Compact Riemannian Manifolds}, Journal of Geometric Analysis 19, no. 3 (2009), 719-736.

\bibitem{Ya92} D. Yang, {\it Convergence of Riemannian manifolds with integral bounds on curvature. I}, Ann. Scient. \`Ec. Norm. Sup., 25 (1992), 77-105.

\bibitem{TiZh13} G. Tian and Z.L. Zhang, {\it Regularity of K\"ahler-Ricci flows on Fano manifolds}, Acta Math., 216 (2016), 127-176.

\bibitem{TiZh15} G. Tian and Z.L. Zhang, {\it Convergence of K\"ahler-Ricci flow on lower dimensional algebraic manifolds of general type},  Int. Math. Res. Not. IMRN 2016, no. 21, 6493-6511. 

\bibitem{Zhang-Zhu15} Q. Zhang and M. Zhu, {\it Li-Yau gradient bounds under nearly optimal curvature conditions},   arXiv:1511.00791.

\bibitem{Zhang-Zhu2016} Q. Zhang and M. Zhu, {\it Li-Yau gradient bound for collapsing manifolds under integral curvature condition},  Proc. Amer. Math. Soc.  145 (2017), no. 7, 3117-3126. 
\end{thebibliography}
\end{document}